\documentclass[12pt]{article}
\usepackage{amssymb}
\usepackage{amsmath}
\usepackage{endnotes}
\usepackage[doublespacing]{setspace}

\newtheorem{theorem}{Theorem}

\newenvironment{proof}[1][Proof]{\noindent\textbf{#1.} }{\ \rule{0.5em}{0.5em}}
\input{tcilatex}
\let\footnote=\endnote

\begin{document}

\title{Measurement error and deconvolution in spaces of generalized
functions }
\author{Victoria Zinde-Walsh\thanks{%
The support of the Social Sciences and Humanities Research Council of Canada
(SSHRC), the Fonds qu\'{e}becois de la recherche sur la soci\'{e}t\'{e} et
la culture (FRQSC) is gratefully acknowledged. P.C.B.Phillips provided
invaluable encouragement, comments and advice; I also thank J. Galbraith for
his support and comments. The Associate Editor, Y. Kitamura and anonymous
referees provided corrections, comments and suggestions that were extremely
helpful; I am very grateful to them. } \\
\\
McGill University and CIREQ\\
victoria.zinde-walsh@mcgill.ca\\
(514) 398 4834}
\maketitle
\date{}

\begin{center}
\bigskip \pagebreak
\end{center}

Running head: Deconvolution in generalized functions

Victoria Zinde-Walsh

Department of Economics, McGill University

855 Sherbrooke Street West,

Montreal, Quebec, Canada

H3A 2T7

\begin{center}
{\LARGE Abstract}
\end{center}

This paper considers convolution equations that arise from problems such as
measurement error and non-parametric regression with errors in variables
with independence conditions. The equations are examined in spaces of
generalized functions to account for possible singularities; this makes it
possible to consider densities for arbitrary and not only absolutely
continuous distributions, and to operate with Fourier transforms for
polynomially growing regression functions. Results are derived for
identification and well-posedness in the topology of generalized functions
for the deconvolution problem and for some regression models. Conditions for
consistency of plug-in estimation for these models are derived.\pagebreak

\section{Introduction}

The focus of this paper is on convolution equations arising in models with
measurement error. Reviews of measurement error models are in Carroll and
Stefanski (2006), Chen, Hong and Nekipelov (2011), Meister(2009). The
convolution equations that are examined here also arise in other contexts;
various models that go beyond measurement error models are enumerated in
Zinde-Walsh (2012). This paper is devoted to the mathematical treatment of
such equations.

Start with the classical measurement error where the variable of interest $%
x^{\ast }$ is observed with error, $u:$%
\begin{equation}
z=x^{\ast }+u.  \label{sum}
\end{equation}%
Here the density of $x^{\ast }$ is the function of interest, denoted $g;$
the observed $z$ has density $w$; suppose that the measurement/contamination
error $u$ is independent of $x^{\ast }$ and has density $f.$ Then the
convolution equation 
\begin{equation}
g\ast f=w,  \label{conv}
\end{equation}%
holds. When densities exist, $g\ast f$ denotes $\int g(z-u)f(u)du$. In
problems that often arise in image processing, epidemiology, medicine the
error density $f$ could be assumed known, e.g. the error is Gaussian noise.

The assumption of known error distribution may not be realistic; if $f$ is
not known additional conditions are needed to identify $g.$ If another
observation, $x,$ on $x^{\ast }$ is available: $x=x^{\ast }+u_{x}$
conditions under which a unique solution exists for the density of interest, 
$g,$ were given by Kotlyarski (1967) in the case of independence between $%
x^{\ast }$ and $u_{x};$ the approach there was to consider the joint
characteristic function of the two measurements.

The assumption of independence of the error in the second measurement may be
too strong for some applications, for example in Cunha, Heckman and
Schennach (2010) one measurement for the latent variable representing a
skill of a child was constructed from test scores where one could plausibly
assume independence for measurement error, but the extra measurements came
from reports by teachers and parents where such an assumption could be
unrealistic. Assume that for the error of the second measurement, $u_{x},$
the conditional expectation is zero: $E(u_{x}|x^{\ast },u)=0$ (but
heteroscedasticity is not ruled out). Denote by $x_{k}$ the $k-th$ component
of $x=(x_{1},...,x_{d})$ and by $h_{k}(x)$ the function $x_{k}g(x)$, and
assume existence of density weighted conditional moments $w_{2k}(z)=E\left(
w(z)x_{k}|z\right) .$ Then in addition to the equation $\left( \ref{conv}%
\right) $ more convolution equations can be written: 
\begin{equation*}
h_{k}\ast f=w_{2k},\text{ }k=1,...,d.
\end{equation*}%
Indeed (with integration here and everywhere in this paper over the whole
space), 
\begin{equation*}
E(w(z)x_{k}|z)=E(w(z)x_{k}^{\ast }|z)=\int (z_{k}-u_{k})g(z-u)f(u)du.
\end{equation*}%
Therefore a system of convolution equations arises for the unknown function $%
g$ and functions $h_{k},$ where $h_{k}(x)=x_{k}g(x),$

\begin{eqnarray}
g\ast f &=&w_{1};  \notag \\
(h_{k})\ast f &=&w_{2k},\text{ }k=1,...d.  \label{newey2}
\end{eqnarray}

Another model that leads to a system of equations is nonparametric
regression with Berkson error (see, e.g. Meister, 2009 for review). Consider
a nonparametric regression model 
\begin{equation*}
y=g(x)+u_{y},
\end{equation*}%
where $x$ may be correlated with the error, $u_{y},$ but where some
instruments, $z,$ are available such that 
\begin{equation}
z=x+u,  \label{berkson}
\end{equation}%
with $z$ is independent of $u$ (Berkson error) and $E(u_{y}|z)=0.$ Denote
the density of $x$ by $f_{x},$ density of $z$ by $f_{z},$ density of $u$ by $%
f$ $_{u}$ and correspondingly that of $-u$ by $f_{-u}.$ Then equation $%
\left( \ref{berkson}\right) $ gives%
\begin{equation}
f_{x}=f_{z}\ast f_{-u}.  \label{berks1}
\end{equation}%
Additionally if expectation conditional on $z$ exists, using independence
between $z$ and $u,$ 
\begin{equation*}
w(z)=E\left( y|z\right) =E\left( g(x)|z\right) =\int g(z-u)f_{u}(u)du;
\end{equation*}%
then%
\begin{equation}
g\ast f_{u}=w.  \label{berks2}
\end{equation}%
The system of equations $\left( \ref{berks1},\ref{berks2}\right) $ involves
two unknown functions, $f_{u}$ and $g,$ where usually the interest is in the
regression function.

The regression function could have instead of an observable argument, $x,$ a
mismeasured or latent argument, $x^{\ast }.$ The model could provide more
equations if another measurement, $x,$ on $x^{\ast }$ were available.
Consider 
\begin{eqnarray}
y &=&g(x^{\ast })+u_{y};  \label{a} \\
x &=&x^{\ast }+u_{x};  \label{b} \\
z &=&x^{\ast }+u.  \label{c}
\end{eqnarray}%
Here $x,y,z$ are observed; assume that $u$ is a Berkson type\ measurement
error independent of $z$ with (unknown) density $f,$ assume that $%
u_{y},u_{x} $ have zero conditional (on $z$ and the other errors)
expectations. For this model Newey (2001) proposed to consider an additional
equation assuming that the conditional moment $E(g(x)x|z)$ exists. In the
univariate case this gives a system of two equations with two unknown
functions (as discussed in Schennach, 2007, Zinde-Walsh, 2009). Define $%
h(x)=xg(x),$ then 
\begin{eqnarray}
g\ast f &=&w_{1},  \label{newey1} \\
h\ast f &=&w_{2},  \notag
\end{eqnarray}%
with $w_{1}=E(y|z),$ $w_{2}=E(xy|z)$ known.

In the multivariate case for $x=(x_{1},...,x_{d})\in R^{d},$ $%
w_{2k}=E(x_{k}y|z),$ $k=1,...,d$ , $h_{k}(x)=x_{k}g(x)$\ this can be
generalized to equations $\left( \ref{newey2}\right) .$ If the density $%
f_{x^{\ast }}$ of $x^{\ast }$ were of interest, equation $f_{x^{\ast
}}=f_{z}\ast f_{-u}$ could be added.

If the independence between the mismeasured variable and measurement error
holds conditionally then the equation $\left( \ref{conv}\right) $ can be
written for densities conditional on some $x_{c}$, where if such densities
exist this equation is%
\begin{equation*}
\int g((z-u)|x_{c})f(u|x_{c})du=w(z|x_{c}).
\end{equation*}%
Then equation $\left( \ref{conv}\right) $ is defined for functions in spaces
where the dimension or the argument is augmented by the dimension of the
conditioning variable, $x_{c}.$

A common way of providing solutions to $\left( \ref{conv}\right) $ and other
equations is to consider them in some normed function spaces, e.g. of
integrable functions such as \thinspace $L_{p}$ or weighted $L_{p}$ spaces,
e.g. Carrasco, Florens and Renault (2007). Solving the equations is often
done by employing Fourier transforms. Since convolutions and Fourier
transforms can be defined in different spaces, the question is which spaces
are best suited for the problems.

Devroye and Gy\H{o}rfi (1985) view density from the perspective of $L_{1}$
space since the density (when it exists as a function) is absolutely
integrable. However, in various problems of interest density may not exist,
as in cases of measurement error for individuals answering survey questions
(say, about income or consumption) where the probability of truthful
reporting is non-zero and a mass point can arise (Hu, 2008). Density
functions in $L_{1}$ do not necessarily converge even if the corresponding
distribution functions converge uniformly. The way to overcome both the
non-existence of density and convergence problems is to consider density as
a generalized derivative of the distribution function as proposed in
Zinde-Walsh (2008); this is done by defining the distribution function as a
functional on a suitable space of well-behaved differentiable functions so
that with this definition the distribution function inherits the good
properties of the well-behaved functions and becomes differentiable (details
in Zinde-Walsh, 2008, also see section 2.1 below); moreover, in the space of
generalized functions the generalized densities converge if the distribution
functions converge thus the problem of defining the density for a
distribution is well-posed there.

Working in spaces of generalized functions also extends the classes of
regression functions for which solutions can be obtained. In various
applications the object of interest may be represented by "sum of peaks"
function, such as a sum of delta functions see, e.g. Klann, Kuhn, Lorenz,
Maass and Thiele (2007) where applications in astrophysics and mass
spectroscopy are discussed; functions with sparse support or support that
includes isolated points can arise in various applications. If $g$
represents a regression function, applying Fourier transform to the
convolution equations in spaces of functions may require severe restrictions
on the function. For example, in spaces of integrable functions (such as $%
L_{1},L_{2})$ linear and polynomial regression functions as well as
distribution functions in binary choice models would be excluded. Again, a
natural extension is to consider spaces of generalized functions where
Fourier transforms are defined for functions that can grow at polynomial
rates as well as for objects with sparse support. Spaces of generalized
functions were utilized by Klann et al (2007) for sum of peaks regression,
by Zinde-Walsh (2008) for generalized density functions; by Schennach (2007)
and Zinde-Walsh (2009) for the problem in errors in variables univariate
regression model with possible polynomial growth in the function.

This paper examines convolution equations in generalized function spaces.
The interest here focuses on the equation $\left( \ref{conv}\right) $ where
only the function $g$ is unknown (deconvolution) and the system of equations 
$\left( \ref{newey2}\right) $ with two unknown functions. The generalized
functions spaces considered here are described in detail in Schwartz (1966)
and Gel'fand and Shilov (1964).

Much of the paper is devoted to a theoretical development of the problem of
solving these convolution equations in spaces of generalized functions:
existence of a unique solution (identification) and continuity of the
mapping from the known functions to the solution (well-posedness).

The usual blueprint for deconvolution in function spaces works as follows.
Assume: $g\ast f=w$ holds, convolution exists, e.g. for functions in $L_{1},$
say, densities. Fourier transform $\left( Ft\right) $ is defined: for the
function $g,$ $Ft(g)=\int g(x)\exp (ix^{T}\zeta )dx;$ this is a
characteristic function if $g$ is a density. Exchange formula applies: for
Fourier transforms $\gamma =Ft(g);\phi =Ft(f);\varepsilon =Ft(w)$ a
convolution is transformed into product:%
\begin{equation*}
g\ast f=\omega \Longrightarrow \gamma \phi =\varepsilon .
\end{equation*}%
Also, if additionally Fourier transform exists for $h_{k}(x)=x_{k}g(x),$then 
$Ft(h_{k})=-i\frac{\partial }{\partial \xi _{k}}\gamma (\xi );$ denote the
derivative $\frac{\partial }{\partial \xi _{k}}\gamma (\xi )$ by $\gamma
_{k}^{\prime }.$

If $w$ and $f$ in equation $\left( \ref{conv}\right) $ (thus also $%
\varepsilon $ and $\phi )$ are known and $\phi \neq 0$ solve the algebraic
equation : $\gamma =\phi ^{-1}\varepsilon ,$ then apply the inverse Fourier
transform, $Ft^{-1},$ to obtain $g,$ 
\begin{equation*}
g=Ft^{-1}\left( \gamma \right) .
\end{equation*}

When the functions that enter the equations are estimated based on available
data on the observables, the solutions will be stochastic and for
establishing consistency well-posedness of the solutions becomes crucial;
Carrasco et al (2007), An and Hu (2012) discuss well-posedness that applies
to similar problems in various normed spaces, mostly in spaces of integrable
functions, $L_{p}.$

Thus for pursuing a similar approach in spaces of generalized functions the
following questions are addressed here: 1. Under what assumptions do the
convolution equations hold? 2. When is Fourier transform and its inverse
defined? 3. When does the exchange formula hold (and when is a
multiplicative product defined)? 4. When do transformed equations have a
unique solution? 5. When is the problem well-posed, that is the solution
continuously depends on the known generalized functions in the equation?

Once the conditions for identification and well-posedness are established in
the generalized functions space, the question of consistent plug-in
estimation can be examined. Suppose that the known functions, such as the
densities (characteristic functions) of the observables, conditional
expectations of the observables (and their Fourier transforms) in the models
considered are consistently estimated; the solutions to the convolution
equations based on these estimated functions are now random generalized
functions; do they converge in some stochastic sense to the true function in
the topology of generalized functions?

In order to answer this one needs to consider stochastic generalized
functions which represent stochastic functionals on the spaces of well
behaved (differentiable, etc.) functions. These are described in Gel'fand
and Vilenkin (1964) and Koralov and Sinai (2007). \ The question of
consistency requires providing conditions for stochastic convergence of
Fourier transforms and inverse Fourier transforms, of derivatives and of
products for random generalized functions - all the operations that are
involved in solving for the unknown functions.

Section 2 of this paper introduces the spaces of generalized functions
considered here, gives conditions when convolutions are defined for
generalized functions of interest, and when products of Fourier transforms\
are defined. Then the convolution theorem ("exchange formula") is provided
making it possible to transform convolutions into products of Fourier
transforms. Section 3 gives results on existence and uniqueness of solutions
of the transformed equation\ or system of equations. These results give
conditions for identification. General results on well-posedness of the
solutions are proved here in Section 3 (some were previously given in the
working paper Zinde-Walsh, 2009). Section 4 provides stochastic properties
of generalized functions. Section 5 gives conditions for consistent (in the
topology of generalized functions) deconvolution and also for consistent
non-parametric estimation of a regression function in the model $\left( \ref%
{a}-\ref{c}\right) $.

\section{Convolution equations in generalized functions}

\subsection{Spaces of generalized functions}

Many different spaces of generalized functions can be defined; each may be
best suited to some particular class of problems. This paper focuses on well
known classical spaces of generalized functions discussed in the books by
Schwartz\ (1966) and Gel'fand and Shilov (1964); reference is also made to
related spaces as presented e.g. by Sobolev (1992).

Define a space of well behaved functions, sometimes called test functions;
denote a generic such space by $G$. The functions in $G$ are defined on the
real space, $R^{d},$ but may take values in the complex space since we
consider characteristic functions and Fourier transforms that can take
complex values. Two widely used spaces of test functions are $G=D$ and $G=S.$

The space $D$ is the linear topological space of infinitely differentiable
functions each defined on a compact support, so that $D\subset C^{\infty
}(R^{d}),$ where $C^{\infty }(R^{d})$ is the space of all infinitely
differentiable functions; to converge in $D$ the sequence of functions
should be supported on a common bounded set and converge uniformly itself as
well as have uniformly converging derivatives of all orders.

To define the space $S$ first introduce some notation. For any vector of
non-negative integers $m=(m_{1},...m_{d})$ and vector $t\in R^{d}$ denote by 
$t^{m}$ the product $t_{1}^{m_{1}}...t_{d}^{m_{d}}$ and by $\partial ^{m}$
the differentiation operator $\frac{\partial ^{m_{1}}}{\partial x_{1}^{m_{1}}%
}...\frac{\partial ^{m_{d}}}{\partial x_{d}^{m_{d}}}.$ The space $S\subset
C^{\infty }(R^{d})$ of rapidly decreasing functions is then defined as:%
\begin{equation*}
S=\left\{ \psi \in C^{\infty }(R^{d}):\left\vert t|^{m}|\partial ^{l}\psi
(t)\right\vert =o(1)\text{ as }t\rightarrow \infty \right\} ,
\end{equation*}%
for any $d-$dimensional vectors of integers $m,l,$ where $l=(0,...0)$
corresponds to the function itself, $\left\vert t\right\vert $ is the vector
of absolute values of vector $t$, $t\rightarrow \infty $ coordinate-wise;
thus \ the functions in $S$ go to zero at infinity faster than any power as
do their derivatives of any order. A sequence in $S$ converges if in every
bounded region each $\left\vert t|^{l}|\partial ^{k}\psi (t)\right\vert $
converges uniformly.

The generalized functions space $G^{\ast }$ is the dual space, the space of
linear continuous functionals on $G.$ For any $b\in G^{\ast }$ and any $\psi
\in G$ denote the value of the functional $b$ applied to $\psi $ by $\left(
b,\psi \right) .$ The topology is defined by weak convergence: a sequence $%
b_{n}\in G^{\ast }$ converges to $b\in G^{\ast }$ if for any $\psi \in G$
the sequence of the values of the functionals converges: $\left( b_{n},\psi
\right) \rightarrow \left( b,\psi \right) .$

Sobolev (1992) gives a general definition (in 1.8) where he points out a
subtle distinction between the functional and a generalized function. Any
generalized function, $b\in G^{\ast },$ can be defined by an equivalence
class $\{b_{n}\}$ of weakly converging sequences of test functions $b_{n}\in 
$\thinspace $G:$%
\begin{equation*}
b=\left\{ \left\{ b_{n}\right\} :b_{n}\in G,\text{ such that for any }\psi
\in G,\underset{n\rightarrow \infty }{\lim }\int b_{n}(t)\overline{\psi (t)}%
dt=(b,\psi )<\infty \right\} ,
\end{equation*}%
where $\int \cdot dt$ denotes the multivariate integral over $R^{d},$
over-bar indicates complex conjugate for complex-valued functions and $%
(b,\psi )$ provides the value of the functional $b\in G^{\ast }$ for $\psi
\in G.$ However, the same functional can be represented by different
generalized functions corresponding to different spaces $G.$ For example,
consider the $\delta -$function. This is a linear continuous functional on
the space $C^{(0)}$ of continuous functions as well as on $D$ or $S$ and
provides $(\delta ,\psi )=\psi (0);$ it can be represented as an equivalence
class of $\delta -$convergent sequences of continuous functions as well as
of functions from $D$ or $S$. This implies that a generalized function
considered as a functional can sometimes be extended to a linear continuous
functional on a wider space.

Note that $D\subset S$ and thus there is the inclusion of the linear dual
spaces: $S^{\ast }\subset D^{\ast };$ convergence in $S^{\ast }$ of linear
continuous functionals implies their convergence in $D^{\ast },$ however, a
sequence of elements of $S^{\ast }$ that converges in $D^{\ast }$ may not
converge in the topology of $S^{\ast }$ (see the example in section 3.2).

In the terminology of Schwartz (1966) generalized functions are sometimes
called "distributions" and elements of $S^{\ast }$ "tempered distributions";
here we shall call them generalized functions indicating the specific space
considered. In Sobolev (1992, p.59) a diagram shows various chains of
generalized functions spaces embedded in each other; these are spaces of
functionals on spaces of continuously differentiable (of different orders)
functions, continuously differentiable functions with compact support and
Sobolev spaces.

Any locally summable (integrable on any bounded set) function $b(t)$ defines
a generalized function $b$ in $D^{\ast }$ by%
\begin{equation}
(b,\psi )=\int b(t)\psi (t)dt  \label{value}
\end{equation}%
on the space of real-valued test functions or by 
\begin{equation}
(b,\psi )=\int b(t)\overline{\psi (t)}dt;  \label{reg}
\end{equation}%
for complex-valued functions. Any locally summable function $b(t)$ that
additionally satisfies 
\begin{equation}
\int \left( (1+t^{2})^{-1}\right) ^{m}\left\vert b(t)\right\vert dt<\infty 
\text{ }  \label{conds}
\end{equation}%
for some non-negative integer-valued vector\textit{\ }$m=(m_{1},...,m_{d})$
with $\left( (1+t^{2})^{-1}\right) ^{m}$\ denoting the corresponding product 
$\Pi _{i=1}^{d}\left( (1+t_{i}^{2})^{-1}\right) ^{m_{i}}$similarly by (\ref%
{reg}) defines a generalized function $b$ in $S^{\ast };$ such generalized
functions are called regular functions in the space of generalized functions.

Generalized derivatives $\partial ^{m}b,$ $m=(m_{1},...m_{d})$ exist for any
generalized function $b$ in $D^{\ast }$ and $S^{\ast }$ and are defined as
functionals by the value for any function, $\psi ,$ as $(\partial ^{m}b,\psi
)=(-1)^{m}(b,\partial ^{m}\psi )$ (here $\left( -1\right)
^{m}=(-1)^{m_{1}+...+m_{d}}).$ The differentiation operator is continuous in 
$S^{\ast }$and in $D^{\ast }.$

Thus, for example, a generalized density function, $f,$ in the univariate
case is defined as a generalized derivative of the regular distribution
function, $F(x),$ by providing for any real-valued function, $\psi \in G,$
the value 
\begin{equation}
(f,\psi )=-\int F(x)\frac{d\psi }{dx}(x)dx.  \label{dens}
\end{equation}%
Thus defined generalized density is an element of $D^{\ast },$ or of $%
S^{\ast },$ moreover, it continuously depends on the distribution functions
in the topology of either $S^{\ast }$ or $D^{\ast }$ by continuity of the
differentiation operator. E.g. if $F(x)=I(x\geq 0)$ (the indicator function $%
I(\theta )=1$ if $\theta $ is true, zero otherwise), by substituting into $%
\left( \ref{dens}\right) $ we obtain the generalized derivative as the Dirac 
$\delta -$function that provides $\left( \delta ,\psi \right) =\psi (0).$

Any generalized function can be represented as a generalized derivative of
some order of a continuous function; to be an element of $S^{\ast }$ such a
continuos function cannot grow faster than a polynomial. More specifically,
by Schwartz (1966, theorem VI, p.239) any generalized function from $b\in
S^{\ast }$ can be represented as a generalized derivative of a continuous
function: $b=\partial ^{l}c$ for a continuous function $c(x),$ which can be
written as $c(x)=\left( 1+x^{2}\right) ^{\frac{m}{2}}\tilde{c}(x),$ where
the continuous function $\tilde{c}(x)$ is uniformly bounded on $R^{d}$: $%
\left\vert \tilde{c}\right\vert <V.$ Consider some fixed vectors of
non-negative integers, $l$ and $m$ and a bound $V.$ Denote by $S_{l,m}^{\ast
}(V)$ the class of generalized functions, $b,$ that have such a
representation for $l$,$m$ and continuous $\tilde{c}$: $\left\vert \tilde{c}%
(x)\right\vert <V;$ any bounded in the topology of $S^{\ast }$ set of
generalized functions has the representation $S_{l,m}^{\ast }(V)$ (see
Schwartz, 1966, (VII,5;5), p.246-247). The smallest such $l$ is the order of
integration that applied to a generalized function produces a continuous
function and is related to the degree of singularity of the generalized
function, while $m$ characterizes its growth at infinity.

The Fourier transform ($Ft)$ is defined for functions in $D,$ and more
generally in $S;$ it is an isomorphism of the space $S;$ the value of the
Fourier transform for the test function $\psi \in S$ also belongs to $S$ and
its value at any $s\in R^{d}$ is $Ft(\psi )(s)=\int \psi (x)e^{ix^{T}s}dx$ ($%
x^{T}$ denotes transpose); in the dual spaces $D^{\ast }$ and $S^{\ast }$
the Fourier transform is given by $(Ft(b),\psi )=(b,Ft(\psi ))$. The Fourier
transform is an isomorphism of the space $S^{\ast }.$

For the analysis of this paper we mostly consider the functions of interest
in the generalized functions space $S^{\ast },$ because of the fact that
Fourier transform represents an isomorphism which permits to apply it or its
inverse to any element, and the operator $Ft$ and the inverse operator, $%
Ft^{-1},$ are continuous in $S^{\ast }.$

\textbf{Assumption 1. } \textit{The generalized functions defined by the
statistical model are in the space }$S^{\ast }.$

This assumption allows for any distribution function on $R^{d}$ and does not
require existence of density functions, for regression functions this allows
growth at infinity, but limits it by $\left( \ref{conds}\right) ,$ thus
binary choice or polynomially growing regression functions are included, but
not functions of exponential growth. Note that exponentially growing
functions are included in the space $D^{\ast }.$

\subsection{Existence of convolutions; convolution pairs}

The convolution of generalized functions can be defined in different ways
(see, e.g. Schwartz, 1966, p.154 or Sobolev, 1992, p. 63; Gel'fand and
Shilov, 1964, v. I, p.103-104); it does not always have meaning and exists
for specific pairs of mutual convolutors.

Consider the following spaces of test functions and of generalized functions
on $R^{d}:$ $\ D,S,C^{\infty },\mathcal{O}_{M},D^{\ast },S^{\ast },E^{\ast },%
\mathcal{O}_{C}^{\ast },$ where $C^{\infty }=C^{\infty }(R^{d})$ is the
space of infinitely differentiable functions on $R^{d};$ $\mathcal{O}%
_{M}\subset C^{\infty }$ is the subspace of infinitely differentiable
functions with every derivative growing no faster than a polynomial, $%
E^{\ast }$ is the subspace of generalized functions with compact support,
and $\mathcal{O}_{C}^{\ast }$ is the subspace of rapidly decreasing (faster
than any polynomial) generalized functions (Schwartz, 1966, p.244). Table 1
shows pairs of spaces for elements of which convolution is defined (X
indicates that convolution cannot be defined for some pairs of elements of
the spaces); the table entries indicate to which space the element resulting
from the convolution operation belongs. The table is an extended version of
the one in the textbook by Kirillov and Gvishiani (1982, p.102) and
summarizes the well-established results in the literature.

\begin{table*}[tbp]
\caption{The convolution table}
\label{sphericcase}$%
\begin{array}{ccccccc}
g\setminus f & D & S & E^{\ast } & O_{C}^{\ast } & S^{\ast } & D^{\ast } \\ 
D & D & S & D & S & O_{M} & C_{\infty } \\ 
S & S & S & S & S & O_{M} & X \\ 
E^{\ast } & D & S & E^{\ast } & O_{C}^{\ast } & S^{\ast } & D^{\ast } \\ 
O_{C}^{\ast } & S & S & O_{C}^{\ast } & O_{C}^{\ast } & S^{\ast } & X \\ 
S^{\ast } & O_{M} & O_{M} & S^{\ast } & S^{\ast } & X & X \\ 
D^{\ast } & C_{\infty } & X & D^{\ast } & X & X & X%
\end{array}%
$%
\end{table*}

The convolution pairs in the table where convolution is defined all possess
the hypocontinuity property (Schwartz, 1966, p.167, p.247-257).
Hypocontinuity of a bilinear operation means that if one component of a pair
is in a bounded set in $G^{\ast }$ and the other converges to zero in $%
G^{\ast }$, the result of the bilinear operation converges to zero
(Schwartz, 1966, pp.72,73).

Convolution of a pair of arbitrary generalized functions is not always
defined in $S^{\ast }$. Bounded support or at least rapid decline at
infinity is needed for a convolution to exist. For example, convolution of a
constant function with another constant function on $\ R^{1}$ is not
defined. Nevertheless there are pairs of subspaces of generalized functions
beyond those in the Table for which convolution defines a generalized
function; spaces where convolution is defined can be combined. The
convolution is a bilinear operation (Schwartz, 1966, p.157); convolution of
a tensor product of generalized functions on two vector spaces, $%
R^{d_{1}},R^{d_{2}}$ is the tensor product of the convolutions of functions
in each space (Schwartz, 1966, p.158). Moreover convolution of any number of
generalized functions can be defined in $D^{\ast }$ as long as all except
possibly one have compact supports and this operation is associative and
commutative (Schwartz, 1966, p.158); a variable shift or derivative of a
convolution exists and is obtained by a shift or differentiation of any of
the generalized functions entering the convolution (Schwartz, 1966, p.160).

\textbf{Definition }\textit{Call a pair of subspaces of generalized
functions, }$A\subset G^{\ast }$\textit{\ and }$B\subset G^{\ast }$\textit{\
a convolution pair }$(A,B)$\textit{\ if for any }$a\in A,b\in B$\textit{\
convolution }$a\ast b$\textit{\ is defined in }$G^{\ast };$\textit{\ it is a
hypocontinuous operation in the topology of }\thinspace $G^{\ast }.$

Note that if $(A,B)$ is a convolution pair then $\left( A\cup G,B\cup
G\right) $ is also a convolution pair.

All the pairs of spaces in the Table satisfy this definition.

\textbf{Assumption 2. }\textit{The statistical model defines functions, }$g$%
\textit{\ and }$f$\textit{\ in }$G^{\ast }$\textit{\ such that }$g\in
A\subset G^{\ast }$\textit{\ and }$f\in B\subset G^{\ast }$\textit{; the
subspaces }$(A,B)$\textit{\ form a convolution pair.}

This assumption implies that (\ref{conv}) holds; it is often satisfied in
statistical problems. Convolution of generalized density functions exists,
thus $\left( \ref{sum}\right) $ leads to $\left( \ref{conv}\right) $ for all
distributions even when the density functions do not exist in the ordinary
sense. The finite sum of $\delta -$functions considered by Klann et al
(2007) is in $E^{\ast },$ thus convolution with any element of $D^{\ast }$
(or $S^{\ast })$ exists in $D^{\ast }(S^{\ast }).$ Schennach (2007)
considered univariate errors in variables model with instrumental variables
and a regression function, $g,$ bounded by polynomials; these regression
functions are in $S^{\ast },$ convolution with any generalized density
function from $\mathcal{O}_{C}^{\ast }$ exists in $S^{\ast }.$ For
regression functions $g$ in subspaces of $S^{\ast }$ where growth is more
restricted, convolution with less rapidly declining $f$ may exist.

For a generalized function $g$ denote by $h_{k}$ the product of $g$ by the $%
k $th component of $x\in R^{d},$ $x_{k},$ for a regular function $g(x)$ this
is $h_{k}(x)=x_{k}g(x)$. In many cases when the convolution $g\ast f$ is
defined, the convolution $h_{k}\ast f$ is also defined. Indeed, this is so
in all the examples above: if $g$ has compact support, so does any $h_{k};$
if $g\in C^{\infty }$, it is true of $h_{k}$ as well, etc. Thus some models
accommodate not only (\ref{conv}), but also other equations, e.g. providing (%
\ref{newey2}).

\textbf{Assumption 3.}\textit{\ The statistical model is such that in
addition to Assumption 2, }$h_{k}\in A$\textit{, }$k=1,...,d.$

\subsection{Fourier transforms, exchange formula and some special
convolution pairs}

Next, consider the Fourier transforms and the exchange formula.

For the generalized functions in equations (\ref{newey2}) denote Fourier
transforms as $\gamma =Ft(g),$ $\phi =Ft(f)$ and $\varepsilon
_{i}=Ft(w_{i}). $ Recall that the Fourier transform always exists in $%
S^{\ast }$ and is a continuous isomorphism.

The classical case of the convolution pair ($S^{\ast },\mathcal{O}_{C}^{\ast
})$ is examined in Schwartz (1966). If $(g,f)$ belong to ($S^{\ast },%
\mathcal{O}_{C}^{\ast })$ the convolution exists and equation (\ref{conv})
transforms into the multiplicative equation $\gamma \phi =\varepsilon ,$
where $\gamma \in S^{\ast },$ $\phi \in \mathcal{O}_{M}$ and $\varepsilon
\in S^{\ast }$ (Schwartz, 1966, p.281-282). The product between a
generalized function in $S^{\ast }$ and a function from $\mathcal{O}_{M}$
always exists in $S^{\ast };$ multiplication is a hypocontinuous operation
(Schwartz, 1966, p.243-246). Thus there is a dichotomous relation between
the classical convolution pair $\left( S^{\ast },\mathcal{O}_{C}^{\ast
}\right) $ and the product pair of generalized functions spaces, $\left(
S^{\ast },\mathcal{O}_{M}\right) .$ Below we show that this dichotomy
extends to other convolution pairs of spaces.

Define a product pair of spaces as a pair $\left( \mathit{\Gamma ,\Phi }%
\right) $ of subspaces of $S^{\ast }\,$such that for any $\gamma \in \Gamma
,\phi \in \Phi $ the product $\gamma \phi $ defines an element $\varepsilon $
in $S^{\ast };$ the operation of multiplication for $\left( \mathit{\Gamma
,\Phi }\right) $ is hypocontinuous in the topology of $S^{\ast }.$

\begin{theorem}
If Assumptions 1 and 2 are satisfied then for any $\left( g,f\right) \in
(A,B)$ the exchange formula applies.
\end{theorem}

\begin{proof}
Consider the special sequences\ studied by Mikusinski (Antosik et al, 1973)
and Hirata and Ogata (1958) that are defined for a generalized function $b$
as $\tilde{b}_{n}=b\ast \delta _{n}$ with $\delta _{n}$ representing the
following delta-convergent sequence: for a number sequence: $\alpha _{n}>0$
and $\alpha _{n}\rightarrow 0,$ the regular function $\delta _{n}(x)$ is
non-negative with support in $\left\vert x\right\vert <\alpha _{n}$ and $%
\int \delta _{n}(x)dx=1;$ the convolution with $\delta _{n}\in E^{\ast }$ is
always defined. Moreover, $\tilde{b}_{n}\rightarrow b$ in $S^{\ast }:$
indeed, $\left( b\ast \delta _{n},\psi \right) =\left( b,\psi \ast \delta
_{n}\right) ,$ but $\psi \ast \delta _{n}$ is in $S$ and converges there to $%
\psi ,$ so $\left( b\ast \delta _{n},\psi \right) $ converges to $\left(
b,\psi \right) .$

Since $g,f\,\ $belong to a convolution pair of spaces $\left( A,B\right) $,
we can enlarge the convolution pair $\left( A,B\right) $ to include with
every $b\in A$ or $B$ the corresponding sequences $\tilde{b}_{n}.$ Denote
the resulting pair of spaces by $\left( \tilde{A},\tilde{B}\right) .$ Show
that this is also a convolution pair. Note that the convolution $\tilde{g}%
_{n}\ast \tilde{f}_{n}$ is defined as support of $\delta _{n}$ is bounded.
All that is needed is to show hypocontinuity; it follows from the fact that
if a~set $T$ from $A$ or $B~$is bounded in $S^{\ast }$ so is the
corresponding set $\tilde{T}$ that contains every element $b\in T$ as well
as all $\tilde{b}_{n}.$ If a sequence $b_{m}$ converges to zero in $S^{\ast
} $ so does the corresponding $\tilde{b}_{mn},$ and hypocontinuity extends
to the enlarged convolution pair.

To show the exchange formula we need to establish that for any $\left(
g,f\right) \in \left( A,B\right) $ we get that $Ft\left( g\right) Ft\left(
f\right) =Ft\left( g\ast f\right) .$

Start by the exchange formula of Hirata and Ogata for the sequences:

\begin{equation*}
\lim Ft(\tilde{g}_{n}\ast \tilde{f}_{n})=\lim Ft(\tilde{g}_{n})\lim Ft(%
\tilde{f}_{n}).
\end{equation*}%
Consider first the left-hand side; by the continuity of Fourier transform in 
$S^{\ast }$ 
\begin{equation*}
\lim Ft(\tilde{g}_{n}\ast \tilde{f}_{n})=Ft(\lim (\tilde{g}_{n}\ast \tilde{f}%
_{n})),
\end{equation*}%
then by hypocontinuity of the convolution and because $\tilde{f}%
_{n}\rightarrow f,$ $\tilde{g}_{n}\rightarrow g$ we get that this is $%
Ft\left( g\ast f\right) .$

On the right hand side by continuity of Fourier transform and convergence 
\begin{equation*}
\lim Ft(\tilde{g}_{n})\lim Ft(\tilde{f}_{n})=Ft\left( g\right) Ft\left(
f\right) .
\end{equation*}%
Denote the space of Fourier transforms of elements from $A$ $\left( \text{or 
}B\right) $ by $Ft\left( A\right) $ (correspondingly, $Ft\left( B\right) ).$
Then $\left( Ft\left( A\right) ,Ft\left( B\right) \right) $ is a product
pair. Hypocontinuity of the product follows immediately from the
hypocontinuity of the convolution and continuity of the Fourier transform
and its inverse.
\end{proof}

\textbf{Corollary 1.} \textit{Given a product pair of spaces }$\left( \Gamma
,\Phi \right) $\textit{\ in }$S^{\ast }$\textit{\ the exchange formula
applies to any }$\gamma ,\phi \in \left( \Gamma ,\Phi \right) :$%
\begin{equation*}
Ft^{-1}\left( \gamma \phi \right) =Ft^{-1}\left( \gamma \right) \ast
Ft^{-1}\left( \phi \right) .
\end{equation*}%
The proof follows the proof of Theorem 1 above by replacing convolution with
product and Fourier transform by the inverse Fourier transform.

In examining problems such as deconvolution much of the literature focuses
on Fourier transforms, e.g. characteristic functions. The dichotomy between
convolution pairs of spaces and product pairs of spaces in $S^{\ast }$
allows to switch between the two types of pairs.

If $\phi $ is the characteristic function of a measurement or contamination
error the condition $\phi \in \mathcal{O}_{M}$ for the classical product
pair it would require existence of all moments. It may be of interest to
consider pairs where products of Fourier transforms of generalized functions
exist for less smooth functions (e.g. with relaxed moments requirements on
measurement error); relaxing the smoothness of $\phi $ will require
restricting the degree of singularity of $\gamma .$

For a continuous function $\phi \in C^{(0)}$ define the space of test
functions, denoted $G\oplus \phi G,$ that consists of functions that can be
represented as $\psi _{1}+\phi \psi _{2},$ with $\psi _{i}\in G$ ($G=D$ or $%
S)$. Consider $\gamma \in G^{\ast }$ that can be extended to a continuous
linear functional on $G\oplus \phi G,$ denote the linear space generated by
such $\gamma $ by $G(\phi )^{\ast }.$ If $\phi \in G$ then $G(\phi )^{\ast
}=G^{\ast }$. For any $\psi \in G$ the value $(\gamma ,\phi \psi )$ is
defined and is a continuous functional with respect to $\psi ;$ this defines
the generalized function $\phi \gamma :$ \ $(\phi \gamma ,\psi )=(\gamma
,\phi \psi ).$ Then $\left( G\left( \phi \right) ^{\ast },G\oplus \phi
G\right) $ is a product pair; the corresponding spaces of inverse Fourier
transforms form a convolution pair.

For example, the derivative, $\delta ^{\prime },$ of a univariate Dirac $%
\delta -$function in $G^{\ast }$ can be multiplied by any continuous
function, $\phi ,$ that is differentiable at 0, since then $(\delta ^{\prime
}\phi ,\psi )=(\delta ^{\prime },\phi \psi )$ with $(\delta ^{\prime },\phi
\psi )=\phi ^{\prime }(0)\psi (0)+\phi (0)\psi ^{\prime }(0).$ More
generally, for a product between a continuous function and a generalized
function to be defined, there is a trade-off between differentiability of
the continuous function and the degree of singularity of the generalized
function.

Under the Assumptions 1-3 the convolution equations (\ref{conv},\ref{newey2}%
) lead to corresponding equations for Fourier transforms: 
\begin{equation}
\gamma \cdot \phi =\varepsilon  \label{eq1}
\end{equation}%
or to the system of equations%
\begin{eqnarray}
\gamma \cdot \phi &=&\varepsilon _{1}  \label{eq2} \\
\gamma _{k}^{\prime }\cdot \phi &=&i\varepsilon _{2k},\text{ }k=1,...,d. 
\notag
\end{eqnarray}

\section{Solutions to the convolution equations: identification and
well-posedness}

\subsection{Identification}

For identification the supports of the functions in the equations play an
important role.

Recall that for a continuous function $\psi (x)$ on $R^{d}$ support is
defined as the set $W=$supp($\psi ),$ such that 
\begin{equation*}
\psi (x)=\left\{ 
\begin{array}{cc}
a\neq 0 & \text{for }x\in W \\ 
0 & \text{for }x\in R^{d}\backslash W.%
\end{array}%
\right.
\end{equation*}%
Support of a continuous function is an open set.

Since generalized functions can be considered as functionals on the space $S$
support of a generalized function $b\in S^{\ast }$ is defined as follows
(Schwartz, 1966, p. 28). Denote by $\left( b,\psi \right) $ the value of the
functional $b$ for $\psi \in S.$ Consider open sets $W$ with the property
that for any $\psi \in S:$ supp$\left( \psi \right) =W$ the value of the
functional $\left( b,\psi \right) =0$; then define the null set for $b$ as
the union of all such sets: $\Omega =\cup W.$ Then supp$\left( b\right)
=R^{d}\backslash \Omega .$ Note that a generalized function has support in a
closed set, for example, support of the $\delta -$function is just one point
0.

We start with deconvolution for (\ref{conv}). In the deconvolution problem $%
f,$ or equivalently its Fourier transform, $\phi ,$ is assumed to be given.
Typically, $\phi $ is a characteristic function and thus is continuous and
bounded.

The convolution equation (\ref{conv}) uniquely identifies $g$ for a known $f$
if it can be shown that the corresponding equation (\ref{eq1}) has meaning
and can be uniquely solved for $\gamma .$

Next a useful Lemma is proved. This Lemma shows that if a product between a
generalized function $\gamma $ and a continuous function $\phi $ is defined: 
$\varepsilon =\gamma \phi ,$ then division of $\varepsilon $ by $\phi $ is
uniquely defined on supp$\left( \phi \right) $ in $D^{\ast }.$

\textbf{Lemma 1. }\textit{Suppose that }$\varepsilon =\gamma \phi $ in $%
D^{\ast };$ $\phi $ \textit{is a continuous function. Then the generalized
function }$\phi ^{-1}\varepsilon $\textit{\ is uniquely defined in }$D^{\ast
}$ \textit{on supp}$\left( \phi \right) .$

Proof. Denote $W=$supp$\left( \phi \right) $ and consider $D\left( W\right) $
the space of all the functions in $D$ with supports restricted to belong to $%
W$; the space $D^{\ast }\left( W\right) $ is the dual space for $D\left(
W\right) .$ Next, consider a covering of the open set $W$ by bounded sets: $%
W=\cup W_{\nu }$ where each $W_{\nu }$ is an open bounded set. Similarly,
consider $D(W_{\nu }).$ Then any generalized function in $D^{\ast }$ can be
restricted to the dual space $D^{\ast }(W_{\nu }).$

In $D^{\ast }(W_{\nu })$ the generalized function $\gamma $ solves $%
\varepsilon =\gamma \phi .$ Suppose that the solution is not unique and
there is $\tilde{\gamma}\neq \gamma $ such that $\tilde{\gamma}\phi
=\varepsilon .$ Then $\left( \gamma -\tilde{\gamma}\right) \phi $ is defined
and represents a zero element in $D^{\ast }\left( W_{\nu }\right) .$ A zero
functional can be extended from the space $D\left( W_{\nu }\right) $ to a
zero functional on the space of continuous functions on $W_{\nu },$ $%
C^{\left( 0\right) }\left( W_{\nu }\right) \subset C^{\left( 0\right)
}\left( R^{d}\right) .$ Then ($\left( \gamma -\tilde{\gamma}\right) \phi
,\psi )=0$ for any $\psi \in C^{\left( 0\right) }\left( W_{\nu }\right) ,$
but since $\phi $ is invertible as a continuous function on $W_{v}$ any
continuous function $\psi $ in $C^{\left( 0\right) }\left( W_{\nu }\right) $
has the representation $\phi \tilde{\psi}$ for some $\tilde{\psi}\in
C^{\left( 0\right) }\left( W_{\nu }\right) .$ This implies that $\left(
\gamma -\tilde{\gamma},\psi \right) $ is defined (equals zero) for any $\psi 
$ thus the functional $\gamma -\tilde{\gamma}$ extends to $C^{\left(
0\right) }\left( W_{\nu }\right) $ as a zero functional. This is only
possible if $\gamma -\tilde{\gamma}$ is a zero functional in $D^{\ast
}(W_{\nu }).$ Then $\gamma -\tilde{\gamma}$ is a zero generalized function
in $D^{\ast }(W).\blacksquare $

Recall that the difference between $D^{\ast }$ and $S^{\ast }$ is in the
"tail behavior" only; on bounded sets the two spaces coincide. The following
theorem provides deconvolution in $S^{\ast }.$

\begin{theorem}
Under Assumptions 1 and 2 assume that $\phi =Ft(f)$ is a known continuous
function; then for any $W$ where supp($\phi )\supset W,$ the Fourier
transform of $g,$ $\gamma ,$ is uniquely defined on $W;$ if it is further
known that supp($\gamma )=W,$ then $g$ is uniquely defined in $S^{\ast }$.
Uniqueness holds automatically if supp($\phi )=R^{d}.$
\end{theorem}

\begin{proof}
By Theorem 1, (\ref{eq1}) holds in $S^{\ast }$. Since supp($\phi )\supset W$
by Lemma 1 a unique solution $\gamma $ to (\ref{eq1}) exists in $D^{\ast
}\left( W\right) $. Since $\gamma $ is the Fourier transform of $g\in
S^{\ast }$, $\gamma \in S^{\ast }$ and defines an element in $S^{\ast
}(W)\subset D^{\ast }\left( W\right) $ uniquely. If support of $\gamma $ is
restricted to $W$ a priori, then the solution is unique, thus when supp($%
\phi )=R^{d},$ the generalized function $\gamma $ is defined uniquely in $%
S^{\ast }.$ An inverse Fourier transform exists in $S^{\ast }$ for any $%
\gamma $. It is then possible to recover $g$ by the inverse Fourier
transform $g=Ft^{-1}(\gamma )$, thus $g$ is uniquely defined whenever $%
\gamma $ is.
\end{proof}

\bigskip Generally by this Theorem identification in deconvolution holds
whenever supp($\phi )=R^{d}.$ However, this excludes some error
distributions such as the uniform or the triangular, where the
characteristic function has isolated zeros. A more general result (Schwartz,
1966, pp.123-125) establishes the possibility of division by $\phi $ even
when $\phi $ has zeros. In the one-dimensional case as long as $\phi $ is
infinitely differentiable, the zeros are isolated and there exists a finite
order derivative that is non-zero at every zero point of $\phi ,$ any
generalized function can be divided by $\phi .$ Thus since the uniform and
triangular distributions have this property, deconvolution with these error
distributions is also identified in $S^{\ast }.$

The next Theorem examines identification in the case when there are two
unknown functions in the system $\left( \ref{newey2}\right) ;$ under
Assumptions 1-3 it leads to $\left( \ref{eq2}\right) $. The main conditions
are continuous differentiability of one of $\gamma ,$ or $\phi ,$
assumptions about support and knowledge of the value of the differentiable
function at an interior point. If that function is a characteristic
function, its value at 0 is always 1.

\begin{theorem}
Under Assumptions 1-3 for the system of equations $\left( \ref{eq2}\right) $
if supp($\phi )\supset $supp($\gamma )=W,$\ where $W$ is a connected set in $%
R^{d}$ that includes $0$ as an interior point, and

(a) if $\gamma $ is continuously differentiable in $W,$ $\gamma (0)=c,$ then 
$\gamma $ is uniquely defined on $W$ by%
\begin{equation}
\gamma (s)=c\exp \int_{0}^{s}\tsum {}_{k=1}^{d}\varkappa _{k}(t)dt_{k},
\label{sola}
\end{equation}%
with the uniquely defined continuous functions $\varkappa _{k}$ that solve 
\begin{equation*}
\varkappa _{k}\varepsilon _{1}-i\varepsilon _{2k}=0,k=1,...,d;
\end{equation*}%
or

(b) if $\phi $ is continuously differentiable in $W$, $\phi (0)=1$, then $%
\gamma $ is uniquely defined on $W$ by 
\begin{equation}
\gamma =\tilde{\phi}^{-1}\varepsilon _{1},  \label{solb}
\end{equation}%
where 
\begin{equation*}
\tilde{\phi}(s)=\exp \int_{0}^{s}\tsum {}_{k=1}^{d}\tilde{\varkappa}%
_{k}(t)dt_{k},
\end{equation*}%
with the uniquely defined on $W$ continuous functions $\tilde{\varkappa}_{k}$
that solve 
\begin{equation*}
\varepsilon _{1}\tilde{\varkappa}_{k}-(\left( \varepsilon _{1}\right)
_{k}^{\prime }-i\varepsilon _{2k})=0,k=1,...,d.
\end{equation*}%
Then $g=Ft^{-1}(\gamma ).$ If supp($\gamma )=W,$ then $g$ is uniquely
defined. Uniqueness holds automatically if $W=R^{d}.$ If supp($\phi )=W,$
then $\phi $ is also uniquely defined; and so is $f=Ft^{-1}(\phi ).$
\end{theorem}

\begin{proof}
(a) Consider the space of generalized functions $D^{\ast }(W)$ (defined in
proof of Lemma 1). Since continuous $\gamma $ is non-zero on $W$ by Lemma 1
(reversing the roles of $\phi $ and $\gamma $ there) in $D^{\ast }\left(
W\right) $ the generalized function $\phi $ is uniquely expressed as $\gamma
^{-1}\varepsilon _{1}.$ Substitute this expression for $\phi $ into the
differential equations in (\ref{eq2}), and denote the continuous functions $%
\gamma _{k}^{\prime }\gamma ^{-1}$ by $\varkappa _{k}(\xi )$ to obtain
equations 
\begin{equation}
\varkappa _{k}(\xi )\varepsilon _{1}-i\varepsilon _{2k}=0,k=1,...,d
\label{kappaa}
\end{equation}%
where the left-hand side is defined and equals zero in $S^{\ast }.$

We can show that the function $\varkappa _{k}$ is uniquely determined in the
class of continuous functions on $W$ by the equation (\ref{kappaa})$.$ Proof
is by contradiction. Suppose that there are two distinct continuous
functions on supp($\gamma ),$ $\varkappa _{k1}\neq \varkappa _{k2}$ that
satisfy ((\ref{kappaa})). Then $\varkappa _{k1}(\bar{x})\neq \varkappa _{k2}(%
\bar{x})$ for some $\bar{x}\in supp(\gamma ).$ Without loss of generality
assume that $\bar{x}$ is in the interior of $W;$ by continuity $\varkappa
_{k1}\neq \varkappa _{k2}$ everywhere for some closed convex $U\subset W.$
Consider now $D(U)^{\ast };$ we can write

\begin{equation*}
(\varepsilon _{1}(\varkappa _{k1}-\varkappa _{k2}),\psi )=0
\end{equation*}%
for any $\psi \in D(U).$ A generalized function that is zero for all $\psi
\in D(U)$ coincides with the ordinary zero function on $U$ and is also zero
for all $\psi \in D_{0}(U)$, where $D_{0}(U)$ denotes the space of
continuous test functions on $U.$ For the space of test functions $D_{0}(U)$
multiplication by continuous $(\varkappa _{k1}-\varkappa _{k2})\neq 0$ is an
isomorphism. Then we can write%
\begin{equation*}
0=(\left[ \varepsilon _{1}(\varkappa _{k1}-\varkappa _{k2})\right] ,\psi
)=(\varepsilon _{1},(\varkappa _{k1}-\varkappa _{k2})\psi )
\end{equation*}%
implying that $\varepsilon _{1}$ is defined and is a zero generalized
function in $D_{0}(U)^{\prime }.$ If that were so $\varepsilon _{1}$ would
be a zero generalized function in $D(U)^{\ast }$ since $D(U)\subset D_{0}(U)$
but this is not possible since $\varepsilon _{1}=\gamma \phi ,$ which is
non-zero by assumption$.$

Next we show that $\gamma $ is then uniquely determined on $W.$ Indeed, for
any $t,s\in $supp($\gamma )$ with the $kth$ coordinate denoted $t_{k},$
write the continuous function 
\begin{equation*}
\gamma (s)=c\exp \int_{0}^{s}\tsum {}_{k=1}^{d}\varkappa _{k}(t)dt_{k},
\end{equation*}%
where integration is along any arc joining $0$ and $s$ in $W.$ This is the
unique solution to $\gamma (0)=c,$ $\gamma ^{-1}\gamma _{k}^{\prime
}=\varkappa _{k}$ (see, e.g., Schwartz, 1966, p.61). Then in $S^{\ast }$ $%
g=Ft^{-1}(\gamma )$ is uniquely defined.

(b) In view of the result in Theorem 2 it is sufficient to show that $\phi $
is uniquely determined on $W.$ Consider the space of generalized functions $%
D^{\ast }(W)$. Since $\phi $ is non-zero on $W$ and continuously
differentiable, then by differentiating the first equation in (\ref{eq2}),
substituting from the second equation and multiplying by $\phi ^{-1}$ in $%
D^{\ast }(W)$ (where the product exists as shown in Lemma 1) we get that the
generalized function%
\begin{equation*}
\varepsilon _{1}\phi ^{-1}\phi _{k}^{\prime }-(\left( \varepsilon
_{1}\right) _{k}^{\prime }-i\varepsilon _{2k})
\end{equation*}%
equals zero in the sense of generalized functions, in $D^{\ast }(W).$ Note
that by assumption $\varepsilon _{1}$ cannot be zero on $W$ and both $%
\varepsilon _{1}$ and $\varepsilon _{2k}~$are zero outside of $W.$ Define $%
\tilde{\varkappa}_{k}=\phi _{k}^{\prime }\phi ^{-1};$ then $\tilde{\varkappa}%
_{k}$ is continuous on supp($\gamma )$ and is a regular function that
satisfies the equation 
\begin{equation}
\varepsilon _{1}\tilde{\varkappa}_{k}-(\left( \varepsilon _{1}\right)
_{k}^{\prime }-i\varepsilon _{2k})=0.  \label{kappa}
\end{equation}%
We can show that the function $\tilde{\varkappa}_{k}$ is uniquely determined
in the class of continuous functions on $W$ by the equation (\ref{kappa})$.$
Proof is identical to the proof in part (a) for equation $\left( (\ref%
{kappaa})\right) .$ Next we show that $\phi $ is then uniquely determined on 
$W.$ Indeed, the continuous function 
\begin{equation*}
\tilde{\phi}(\zeta )=\exp \int_{0}^{s}\tsum {}_{k=1}^{d}\varkappa
_{k}(t)dt_{k},
\end{equation*}%
is the unique solution to $\tilde{\phi}(0)=1,$ $\tilde{\phi}_{k}^{-1}\tilde{%
\phi}^{\prime }=\varkappa _{k}$; then since $\varkappa _{k}(=\phi
_{k}^{\prime }\phi ^{-1})$ is uniquely determined on $W,$ so is $\phi $ on $%
W $ where it coincides with $\tilde{\phi}.$

By Theorem 2, then $\gamma $ is then uniquely defined on $W.$ Then in $%
S^{\ast }$ $g=Ft^{-1}(\gamma )$ is uniquely defined.
\end{proof}

This Theorem extends the identification results of Schennach (2007) and
Zinde-Walsh(2009) to the multivariate case and to the case when $\gamma $
rather than $\phi $ is continuously differentiable (part (a) of this
Theorem), and extends the identification result of Cunha et al (2011) by
showing that in the model with several measurements considered there
identification additionally holds with the requirement of continuous
differentiability of $\phi $ replacing the requirement that $\gamma $ be
continuously differentiable (part (b)).

\subsection{Well-posedness of the deconvolution in $S^{\ast }$}

Well-posedness requires that a unique solution to the problem exist and that
this solution be continuous in some "reasonable topology" (Hadamard, 1923).
Here most of the results consider the topology of generalized functions
which is weaker than, say, the uniform or $L_{1}$ norm for corresponding
subspaces, thus well-posedness may hold in this topology when it does not
hold in the usual norms. But well-posedness may not always obtain even in
this weak topology. An example is provided that involves deconvolution with
a supersmooth distribution.

Consider a convolution pair of spaces $\left( A,B\right) $ in $S^{\ast },$
denote by $C$ the space of elements in $S^{\ast }$ represented by
convolutions of generalized functions from $\left( A,B\right) .$
Well-posedness of the deconvolution would require that for the given $f\in B$
and any sequence of $w_{n}\in C$ that converges to some $w\in C$ in the
topology of $S^{\ast }:w_{n}\rightarrow w,$ the sequence of $g_{n}\in A$
such that $g_{n}\ast f=w_{n}$ would converge to $g\in A,$ with $g\ast f=w.$
By Theorem 1\ this convergence can be restated for the product pair of
spaces $\left( \Gamma ,\Phi \right) $ where the product pair is the image of
the convolution pair $\left( A,B\right) $ under Fourier transform; denote by 
$\Pi \in S^{\ast }$ the space of products of elements from the pair $\left(
\Gamma ,\Phi \right) $. Then well-posedness can be restated in terms of the
Fourier transforms: the problem is well-posed in $S^{\ast }$ if for any $%
\varepsilon _{n}\rightarrow \varepsilon $ where $\varepsilon
_{n},\varepsilon \in \Pi $ the corresponding sequence $\gamma _{n}=\phi
^{-1}\varepsilon _{n}$ converges to $\gamma =\phi ^{-1}\varepsilon .$

If $\phi ^{-1}\in O_{M}$ then by hypocontinuity of the product, for any
sequence $\varepsilon _{n}-\varepsilon $ that converges to zero in $S^{\ast }
$ the corresponding sequence $\gamma _{n}-\gamma $ also converges to zero
and well-posedness obtains without any restrictions on $\gamma $. This
result could also extend to infinitely differentiable $\phi $ with some
zeros and applies in the univariate deconvolution to error distributions
such as the uniform and triangular (see Schwartz, pp. 123-125).

Below we provide cases of well-posed deconvolution where less restrictive
differentiability conditions are imposed on $\phi ;$ this would require
additional conditions on the $\gamma $'s. The nature of the conditions is to
ensure that the product pair is defined for the Fourier transforms; for a
continuous $\phi $ this requires a trade-off between the degree of
singularity of $\gamma $ (and correspondingly, $\varepsilon )$ and the
differentiability of $\phi ,$ these trade-offs can occur locally as in the
example where $\gamma $ is the derivative of a $\delta -$function and $\phi $
is continuously differentiable at 0, but to streamline the proofs we
consider global restrictions in the product pairs.

First, suppose that both $\phi $ and $\gamma $ are continuos functions, then 
$\varepsilon $ is continuous as well. Let $\Gamma $ be a subspace of all
continuous functions on $R^{d}$ such that all $\gamma \in \Gamma $ belong in
a bounded set in $S^{\ast };$ then for some $\bar{m}$ we have $\Gamma =S_{0,%
\bar{m}}^{\ast }(V)$ (implying $\left( \left( 1+t^{2}\right) ^{-1}\right) ^{%
\bar{m}}\left\vert \gamma \right\vert <V<\infty ).$ Then for any bounded
continuous $\phi $ the product $\varepsilon =\gamma \phi \in S_{0,\bar{m}%
}^{\ast }(V)$\textit{. }

\textbf{Lemma 2.} \textit{Suppose that }$\gamma ,\gamma _{n}\in S_{0,\bar{m}%
}^{\ast }(V),$ $\phi $ \textit{is a bounded continuous function; }$%
\varepsilon _{n}=\gamma _{n}\phi ,\varepsilon =\varepsilon _{0}=\gamma \phi
,0\leq n<\infty $ \textit{\ and }$\phi ^{-1}$\textit{\ is a continuous
function that satisfies (\ref{conds}). Then if }$\varepsilon _{n}\rightarrow
\varepsilon _{0}$\textit{\ as }$n\rightarrow \infty $ \textit{in }$S^{\ast }$%
\textit{\ we get that }$\gamma _{n}=\phi ^{-1}\varepsilon _{n}$\textit{\
converges to }$\gamma =\phi ^{-1}\varepsilon _{0}.$

Proof. First note that for any number $\xi >0$ for any $\psi \in S$ there
exists some bounded set $\Lambda \in R^{d}$ such that 
\begin{equation*}
|\int_{R^{d}\smallsetminus \Lambda }\varepsilon _{n}(t)\phi ^{-1}(t)\psi
(t)dt|<\xi .
\end{equation*}%
Indeed, $|\int_{R^{d}\smallsetminus \Lambda }\varepsilon _{n}(t)\phi
^{-1}(t)\psi (t)dt|\leq $ 
\begin{eqnarray*}
&&V\int_{R^{d}\smallsetminus \Lambda }\left( \left( 1+t^{2}\right)
^{-1}\right) ^{-(\bar{m}+m)}|\phi ^{-1}(t)||\left( \left( 1+t^{2}\right)
^{1}\right) ^{m}\psi (t)|dt \\
&\leq &V\underset{R^{d}\smallsetminus \Lambda }{\sup }|\left( \left(
1+t^{2}\right) ^{1}\right) ^{m}\psi (t)|\int_{R^{d}}\left( \left(
1+t^{2}\right) ^{-1}\right) ^{-m}|\phi ^{-1}(t)|dt.
\end{eqnarray*}%
Since by (\ref{conds}) $\int_{R^{d}}\left( \left( 1+t^{2}\right)
^{-1}\right) ^{-m}|\phi ^{-1}(t)|dt$ is bounded, say, by some $V_{\phi }$
and for $\psi \in S$ as $t\rightarrow \infty $ the value $|\left( \left(
1+t^{2}\right) ^{1}\right) ^{m}\psi (t)|$ converges to zero, the set $%
\Lambda $ can be selected such that 
\begin{equation*}
\underset{R^{d}\smallsetminus \Lambda }{\sup }|\left( \left( 1+t^{2}\right)
^{-1}\right) ^{m}\psi (t)|<\xi V_{\phi }^{-1}V^{-1}.
\end{equation*}%
Consider now the value of the functional $\left( \varepsilon _{n}\phi
^{-1}-\varepsilon \phi ^{-1},\psi \right) .$ Since the sequence of
continuous bounded functions $\varepsilon _{n}-\varepsilon $ converges to
zero in $S^{\ast }$ it converges to zero point-wise and uniformly on bounded 
$\Lambda $. Then for any $\zeta $ and $\Lambda $ corresponding to $\xi =%
\frac{1}{2}\zeta $ we can find $N$ such that 
\begin{equation*}
\underset{\Lambda }{\sup }\left\vert \varepsilon _{n}-\varepsilon
_{0}\right\vert \left\vert \int \phi ^{-1}\left( t\right) \psi
(t)dt\right\vert <\xi
\end{equation*}
for any $n>N.$ It follows that $\left\vert \left( \phi ^{-1}\varepsilon
_{n}-\phi ^{-1}\varepsilon ,\psi \right) \right\vert <\zeta .$

Thus $\left\vert \gamma _{n}-\gamma \right\vert \rightarrow 0$ in $S^{\ast
}.\blacksquare $

In the measurement error problem the function $\varepsilon $ is a
characteristic function; if all $\varepsilon _{n}$ in the equation $\left( %
\ref{eq1}\right) $ are characteristic functions or at least some continuous
and uniformly bounded functions then they satisfy the conditions of Lemma 2
with $\bar{m}=0$.

The somewhat unusual requirement that the sequence that converges be bounded
in $S^{\ast \text{ }}$is a requirement that is associated with the weak
topology of the space of generalized functions. In normed spaces convergence
implies eventual boundedness in norm of the converging sequence, but this is
not the case for generalized functions that do not exclude singularities and
polynomial growth. The boundedness requirement uniformly limits the degree
of singularity and the rate of divergence at infinity of the generalized
functions in the bounded set. It is however less restrictive than many
assumptions in normed spaces such as boundedness or integrability.

The Lemma demonstrates that unlike the behavior in function spaces where the
deconvolution is usually ill-posed, in the weaker topology of the space $%
S^{\ast }$ deconvolution is well-posed when the error distribution is such
that the condition (\ref{conds}) is satisfied. Usually well-posedness in
deconvolution is examined in terms of ordinary smooth and supersmooth
distributions. When the distribution is ordinary smooth the condition (\ref%
{conds}) holds; it is not satisfied for supersmooth distributions.

An and Hu (20111) demonstrate that well-posedness holds automatically in a
measurement error problem provided there is a mass at zero; this happens if
there is a reporting error but with some non-zero probability some values
are truthfully reported. Indeed if there is a mass at zero the error
distribution is a mixture with the delta-function that has Fourier transform
equal to 1, then $\phi $ is separated from zero and condition (\ref{conds})
for $\phi ^{-1}$ holds.

A further extension is possible to deconvolution where the generalized
function $\gamma $ is not necessarily a continuous function and can be
singular. Then additional differentiability conditions are needed to ensure
that $\gamma $ and $\phi $ are in a product pair. Let $\Gamma =S_{l,m}^{\ast
}(V)$ with $l=\left( l_{1},...,l_{d}\text{ }\right) $; $\phi $ be a bounded
continuous function that has continuous derivatives of any order $\leq
l_{1}+...+l_{d};$ then a product $\varepsilon =\gamma \phi $ exists in $%
S^{\ast }$ for any $\gamma \in \Gamma .$

\begin{theorem}
\textbf{\ }Assume $\gamma ,\gamma _{n}\in S_{l,m}^{\ast }(V),$ and $\phi $
is a continuous regular function that has continuous derivatives $\partial
^{l_{i}}$ for all subvectors $l_{i}$ of $l;$ $\varepsilon =\gamma \phi ,$ $%
\varepsilon _{n}=\gamma _{n}\phi $ for $n=1,2,...;$ and $\varepsilon
_{n}-\varepsilon \rightarrow 0$ in $S^{\ast }.$ If $\phi ^{-1}$ and all $%
\partial ^{l_{i}}\left( \phi ^{-1}\right) $ for any subvector $l_{i}$ of
each $l$ satisfy (\ref{conds}) then $\varepsilon _{n}\phi ^{-1}$ exists in $%
S^{\ast }$ and $Ft^{-1}(\varepsilon _{n}\phi ^{-1})\rightarrow g$ in $%
S^{\ast }.$
\end{theorem}

\begin{proof}
Denote by $c_{n}$ and $c$ the continuous functions for which $\varepsilon
_{n}$ and $\varepsilon $ can be defined via the operator $\partial ^{l}.$
Consider the functional $\left( \phi ^{-1}\varepsilon _{n}-\phi
^{-1}\varepsilon ,\psi \right) =\left( \phi ^{-1}\partial ^{l}\left(
c_{n}-c\right) ,\psi \right) ;$ it can be extended to the functional 
\begin{equation*}
\left( c_{n}-c,\left( -1\right) ^{\left\vert l\right\vert }\partial
^{l}\left( \phi ^{-1}\psi \right) \right) =\left( -1\right) ^{\left\vert
l\right\vert }\tsum\limits_{(l_{1};l_{2})=l}\kappa \left( l_{1},l_{2}\right)
\left( c_{n}-c,\partial ^{l_{1}}(\phi ^{-1})\partial ^{l_{2}}\psi \right) ,
\end{equation*}%
where $\kappa \left( l_{1},l_{2}\right) $ is the corresponding integer
coefficient arising from differentiation of the product. All the functionals 
$\left( c_{n}-c,\partial ^{l_{1}}(\phi ^{-1})\partial ^{l_{2}}\psi \right) $
that enter into the linear combinations on the right-hand side satisfy the
conditions of Lemma 2 and thus converge to zero. Therefore, $\gamma
_{n}-\gamma =(\varepsilon _{n}-\varepsilon )\phi ^{-1}\rightarrow 0$ in $%
S^{\ast }.$

By continuity in $S^{\ast }$ of\ the inverse Fourier transform the limit 
\begin{equation*}
Ft^{-1}(\varepsilon _{n}\phi ^{-1})\rightarrow g
\end{equation*}%
in $S^{\ast }$ follows.
\end{proof}

\bigskip If the measurement error model holds conditionally on some $d_{c}$%
-dimensional $x_{c}$ assume that the all the distribution functions are
defined on $R^{d}\times R^{d_{c}}$ as $F(\cdot ,x_{c}).$ For any generalized
function $b$ on $R^{d_{x}}\times R^{d_{c}}$ the partial Fourier transform $%
Ft_{|x_{c}}$ is defined as follows: for $\psi (x,x_{c})\in S$ with $x\in
R^{d_{x}},x_{c}\in R^{d_{c}}$ define $Ft_{|x_{c}}(\psi )\left(
s,x_{c}\right) \in S$ by $\int e^{ix^{T}s}\psi (x,x_{c})dx,$ then $\left(
Ft_{|x_{c}}b,\psi \right) =\left( b,Ft_{|x_{c}}(\psi )\right) .$ Consider
the following possibilities: 1. the generalized functions $\gamma ,\phi
,\varepsilon $ denote now the partial Fourier transforms or the distribution
functions; 2. instead of distribution functions consider generalized density
functions and denote by $\gamma ,\phi ,\varepsilon $ the corresponding
partial Fourier transforms; 3. assume continuous differentiability of the
probability distribution functions with respect to $x_{c}$ and consider
conditional probability or, correspondingly, conditional generalized
densities; additional conditions will be needed to make sure that dividing
by the density of $x_{c}$ is possible in $S^{\ast }.$ In all these cases $%
\left( \ref{eq1}\right) $ holds. To avoid imposing extra constraints here
examine case 2. For deconvolution $\phi $ is assumed known and often $\phi $
does not vary with \ $x_{c}.$ The conditions of Lemma 2 may not apply as the
functions $\varepsilon ,\varepsilon _{n}$ may no longer be continuous,
however, since a probability distribution is a bounded monotone function, $%
\varepsilon $ is in some set $S_{1,d_{c}}^{\ast }(V)$ and thus theorem 4
requires only that all $\varepsilon _{n}$ belong to $S_{1,d_{c}}^{\ast }(V)$%
, that (\ref{conds}) is satisfied for $\phi ^{-1}$ and that $\phi $ be
differentiable in components of $x.$

The condition (\ref{conds}) for $\phi ^{-1}$ is necessary for
well-posedness: the example below shows that even in the weak topology of
generalized functions well-posedness does not hold if it is not satisfied
e.g. for deconvolution with supersmooth distributions.

\textbf{Example. }\textit{Consider the function }$\phi (x)=e^{-x^{2}},$%
\textit{\ }$x\in R,$\textit{\ then }$\phi ^{-1}$\textit{\ does not satisfy (%
\ref{conds}). Then there exists a sequence of functions }$\gamma _{n}$%
\textit{\ and a }$\gamma $\textit{\ in }$S^{\ast }$\textit{\ such that }$%
\gamma _{n}\phi \rightarrow \gamma \phi $\textit{\ in }$S^{\ast },$\textit{\
but }$\gamma _{n}$\textit{\ does not converge to }$\gamma $\textit{\ in }$%
S^{\ast }.$

\textit{Define the function }$b_{n}(x)=\phi ^{-1}(x)I\left( n-\frac{1}{n},n+%
\frac{1}{n}\right) ;$\textit{\ then }$b_{n}\in S^{\ast }.$\textit{\ Then
define }$\gamma _{n}=\gamma +b_{n}$\textit{\ for some fixed }$\gamma \in
S^{\ast }.$\textit{\ Define }$\varepsilon _{n}=\gamma _{n}\phi $\textit{\
and }$\varepsilon =\gamma \phi ;$\textit{\ since }$\phi \in S,$\textit{\ the
products are always defined. Consider the difference }$\varepsilon
_{n}-\varepsilon ;$\textit{\ it equals }$I\left( n-\frac{1}{n},n+\frac{1}{n}%
\right) .$\textit{\ In }$S^{\ast }$\textit{\ this sequence converges to
zero. On the other hand, }$\left( \gamma _{n}-\gamma ,\psi \right) $\textit{%
\ for any }$\psi \in S$\textit{\ equals }%
\begin{equation*}
\int_{n-2/n}^{n+2/n}e^{x^{2}}\psi (x)dx.
\end{equation*}%
\textit{Select }$\psi \in S$\textit{\ given by }$\psi (x)=\exp (-\left\vert
x\right\vert )$\textit{; then }$\left( \gamma _{n}-\gamma ,\psi \right) =$%
\begin{eqnarray*}
\int_{n-2/n}^{n+2/n}e^{x^{2}}\psi (x)dx &\geq
&\int_{n-1/n}^{n+1/n}e^{x^{2}-x}dx \\
&\geq &\frac{1}{2n}e^{-(n+\frac{1}{n})+\left( n-\frac{1}{n}\right) ^{2}}.
\end{eqnarray*}%
\textit{Thus for this }$\psi $ \textit{the values of the functional, }$%
\left( \gamma _{n}-\gamma ,\psi \right) $\textit{\ diverge as }$n\rightarrow
\infty ,$\textit{\ and so }$\gamma _{n}-\gamma $\textit{\ does not converge
to zero in }$S^{\ast }.\blacksquare $

Note that since any $\psi \in D$ has bounded support in the example
convergence of $\gamma _{n}$ to $\gamma $ holds in $D^{\ast }.$

\subsection{Well-posedness in $S^{\ast }$ of the solution to the system of
equations (\protect\ref{eq2})}

Recall (\ref{conds}) and define a subclass of functions in $S^{\ast },$ $%
\Phi (m,V),$ where $b\in \Phi (m,V)$ if $b$ satisfies the condition

\begin{equation}
((1+t^{2})^{-1})^{m}\left\vert b(t)\right\vert <V<\infty \text{ }
\label{uniform}
\end{equation}

The requirement that $b$ belong to $\Phi (m,V)$ implies that (\ref{conds})
applies uniformly in this class and also defines a bounded set in $S^{\ast
}. $ The next Theorem considers well-posedness of the problem in $\left( \ref%
{newey2}\right) $ (and correspondingly, $\left( \ref{eq2}\right) )$ under
the identification conditions of Theorem 3(b). The conditions of Theorem
3(a) can be similarly considered.

\begin{theorem}
\textbf{\ }Suppose that $(\gamma _{n},\phi _{n}),n=1,2,...$ and $(\gamma
,\phi )$ belong to a product pair $\left( \Gamma ,\Phi \right) .$
Additionally, let the conditions of Theorem 3(b) apply to each pair $(\gamma
_{n},\phi _{n}).$ Suppose that $\varepsilon _{1n}-\varepsilon
_{1}\rightarrow 0$ and $\varepsilon _{2kn}-\varepsilon _{2k}\rightarrow 0$
in $S^{\ast };$ the functions $\phi ,\phi _{n}$ as well as $\phi ^{-1},\phi
_{n}^{-1}$ restricted to $W$ all belong to some $\Phi (m,V)$ Then if

(a) $\left( \Gamma ,\Phi \right) \equiv \left( S^{\ast },O_{M}\right) ,$

or

(b) all $\varkappa _{kn}=\left( \phi _{n}\right) _{k}^{\prime }\phi
_{n}^{-1},$ and $\varkappa _{k}=\left( \phi \right) _{k}^{\prime }\phi ^{-1}$
are such that $\left( \varkappa _{kn},\varepsilon _{1n}\right) ,$ $\left(
\varkappa _{k},\varepsilon _{1}\right) $ belong to some product pair,

the products $\varepsilon _{n}\phi _{n}^{-1}$ exist in $S^{\ast }$ and $%
Ft^{-1}(\varepsilon _{1n}\phi _{n}^{-1})\rightarrow g$ in $S^{\ast }.$
\end{theorem}

\begin{proof}
(a) We have that $\phi ,\phi _{n}\in \mathcal{O}_{M}.$ It follows that $%
\left( \phi \right) _{k}^{\prime },\left( \phi _{n}\right) _{k}^{\prime }\in 
\mathcal{O}_{M}.$ Also $\left( \phi \right) _{k}^{\prime },\left( \phi
_{n}\right) _{k}^{\prime }\in \Phi (m^{\prime },V),$ where $m^{\prime
}=m+\iota ,$ with $\iota $ a vector of ones. From Theorem 3 it follows that
for every $n$ the functions $\gamma _{n}$ and $\phi _{n}$ are uniquely
identified on $W.$ From now on we consider all functions and function spaces
restricted to $W,$ even when $W$ does not coincide with $R^{d},$ but keep
the same notation. The functions belong also to $\Phi (\tilde{m},V)$ where $%
\tilde{m}=m+m^{\prime }.$ Without loss of generality assume that each $%
\varkappa _{k}$ is also in the same $\Phi (\tilde{m},V),$ and so all $%
\varkappa _{kn},\varkappa _{k}$ are in a bounded set in $S^{\ast }$. Since
from condition $\left( a\right) $ it follows that $\varkappa _{kn}=\left(
\phi _{n}\right) _{k}^{\prime }\phi _{n}^{-1}\in \mathcal{O}_{M},$ products
are defined and from equations $\varepsilon _{1n}\varkappa _{kn}-(\left(
\varepsilon _{1n}\right) _{k}^{\prime }-i\varepsilon _{2kn})=0$ and
convergence of $\varepsilon _{in}$ to $\varepsilon _{i}$ we get that $%
\varepsilon _{1n}\varkappa _{kn}-\varepsilon _{1}\varkappa _{k}$ converges
to zero in $S^{\ast }.$ For functions in $\mathcal{O}_{M}$ products with any
elements from $S^{\ast }$ exist, thus $\varepsilon _{1n}\varkappa
_{kn}-\varepsilon _{1}\varkappa _{kn}$ exists; moreover $\left( \varepsilon
_{1n}-\varepsilon _{1}\right) \varkappa _{kn}$ converges to zero in $S^{\ast
}$ by the hypocontinuity property (Schwartz, p.246). It follows that $%
\varepsilon _{1}(\varkappa _{kn}-\varkappa _{k})$ converges to zero in $%
S^{\ast }.$ Since $\varepsilon _{1}$ is supported on $W$ and $(\varkappa
_{kn}-\varkappa _{k})\in \mathcal{O}_{M}$ by continuity of the functional $%
\varepsilon _{1}$ it follows that $\varkappa _{kn}-\varkappa _{k}$ converges
to zero on $W.$ It then follows that $\phi _{n}-\phi \rightarrow 0$ in $%
S^{\ast }$ as well as pointwise and uniformly on bounded sets in $W$, the
product $\phi ^{-1}\phi _{n}^{-1}$ is in a bounded set in $S^{\ast },$ thus $%
\phi _{n}^{-1}-\phi ^{-1}=\phi ^{-1}\phi _{n}^{-1}\left( \phi -\phi
_{n}\right) $ converges to zero in $S^{\ast }.$

Consider $\varepsilon _{1n}\phi _{n}^{-1}-\varepsilon _{1}\phi
^{-1}=\varepsilon _{1n}(\phi _{n}^{-1}-\phi ^{-1})+(\varepsilon
_{1n}-\varepsilon _{1})\phi ^{-1};$ this difference converges to zero in $%
S^{\ast },$ thus $\gamma _{n}$ converges to $\gamma $ in $S^{\ast }$ and
since the Fourier transform is continuously invertible in $S^{\ast },$ $%
g_{n}=Ft^{-1}(\varepsilon _{1n}\phi _{n}^{-1})\rightarrow g$ in $S^{\ast }.$

(b) Modify the proof of $\left( a\right) $ by stating existence of products
based on the assumption in $\left( b\right) $ rather than assuming the
specific product pair $\left( S^{\ast },O_{M}\right) .$
\end{proof}

Thus well-posedness obtains for\ the solution to the errors in variables
regression with Berkson type instruments that provides (\ref{newey2}) as
long as $\phi ^{-1}$ and all the $\phi _{n}^{-1}$ are all in the same class $%
\Phi (m,V));$ e.g. a set of the characteristic functions would satisfy this
if they are not supersmooth and the growth of all $\phi _{n}^{-1}$ is
bounded by the same order polynomial. With the condition (a) which does not
restrict $\gamma \in S^{\ast }$ the Theorem will hold if the supports of $%
\phi ,\phi _{n}$ are bounded or if these were Fourier transforms of
functions with some singularity points (e.g. from distributions with mass
points), when $\phi $ would include a constant. The condition (b) would
require imposing restrictions on degree of singularity of $\gamma $ and
further differentiability conditions on $\phi $ and $\phi _{n}$ of the type
imposed in Theorem 4 to ensure that all the products exist. Even if in the
model $\phi $ is a characteristic function, $\phi _{n}$ need not necessarily
be characteristic functions, but just continuous functions that satisfy the
conditions. This result did not require any restrictions on $\gamma $ and
thus on the regression function in the model beyond belonging to $S^{\ast }.$
Under Theorem 3(a) a similar result holds after imposing the conditions on $%
\gamma ,\gamma _{n}$ rather than $\phi ,\phi _{n}$ to ensure that
generalized functions $\varkappa _{kn}=\left( \gamma _{n}\right)
_{k}^{\prime }\gamma _{n}^{-1}$ belong to $\mathcal{O}_{M}$ and to some $%
\Phi (\tilde{m},V).$

For example, if as in Cunha et al (2010) both $\gamma $ and $\phi $ are
characteristic functions and thus continuous and bounded, products exist and
Theorem 5(b) applies, as long as the conditions of Theorem 3 are satisfied
and the condition $\left( \ref{uniform}\right) $ is satisfied for the
inverses of the functions in the sequence, that is e.g. under Theorem 3(a)
for $\gamma _{n}^{-1},$ or under 3(b) for $\phi _{n}^{-1},$ so that the
corresponding distributions are not supersmooth and the polynomial lower
bound holds uniformly, well-posedness holds.

The implications that well-posedness has for estimation are two-fold. One is
that unless well-posedness holds, that is the inverse mapping from the class
of the known functions into the class of identified solutions is continuous,
the solutions corresponding to the consistently estimated known functions
will not in general provide consistent estimators for the solutions. The
other is that in a well-posed problem consistent estimation of the known
functions automatically gives rise to consistency of plug-in estimators of
solutions; the next section examines consistent estimation of the
deconvolution and the solution to the system (\ref{newey2}) in the space $%
S^{\ast }.$

\section{Random generalized functions and stochastic convergence}

\bigskip This section examines stochastic convergence of the solutions to
the deconvolution equation (\ref{conv}) and to (\ref{newey2}), equivalently
to $\left( \ref{eq1}\right) $ and $\left( \ref{eq2}\right) .$ If \ some
consistent estimators are available for either the function $w$ ($w_{1}$ and 
$w_{2k})$ or, equivalently, for the Fourier transform, $\varepsilon $ ($%
\varepsilon \,_{1}$ and $\varepsilon _{2k}),$ stochastic convergence of the
solutions provides consistency results for plug-in estimators of $g$ as long
as well-posedness holds. Since generalized functions in any space $G^{\ast }$
are represented as linear continuous functionals on $G$ results for random
functionals are applicable here, the specific feature is that for any
generalized function $b\in G^{\ast }$ one needs to consider the collection
of random functionals indexed by the functions from the space $G.$

\subsection{Random generalized functions}

Following Gel'fand and Vilenkin (1964) define random generalized functions
as random linear continuous functionals on the space of test functions (see
e.g. Koralov and Sinai, 2007 who consider specifically $S^{\ast }-$ ch.17).
In particular, any random generalized function $\tilde{b}$ on $G$ is
represented by a collection of (complex-valued) random variables on a common
probability space that are indexed by $\psi \in G,$ denoted $(\tilde{b},\psi
),$ such that

(a) $(\tilde{b},(a_{1}\psi _{1}+a_{2}\psi _{2}))=a_{1}(\tilde{b},\psi
_{1})+a_{2}(\tilde{b},\psi _{2})$ a.s.;

(b) if $\psi _{kn}\rightarrow \psi _{k}$ in $S$ as $n\rightarrow \infty
,k=1,2..,m,$ then vectors 
\begin{equation*}
((\tilde{b},\psi _{1n})...(\tilde{b},\psi _{mn}))\rightarrow _{d}((\tilde{b}%
,\psi _{1})...(\tilde{b},\phi _{m}))
\end{equation*}%
where $\rightarrow _{d}$ denotes convergence in distribution.

As shown in Gel'fand and Vilenkin, equivalently, there exists a probability
measure on $G^{\ast }$ such that for any set $\psi _{1},...\psi _{m}\in G$
the random vectors (($b,\psi _{1}),...(b,\psi _{m}))$ have the same
distribution as for some random functional $\tilde{b},$ $((\tilde{b},\psi
_{1})...(\tilde{b},\psi _{m})).$ An example is a generalized Gaussian
process $b,$ so defined if for any $\psi _{1},...,\psi _{m}$ the joint
distribution of $\left( b,\psi _{1}\right) ,...,\left( b,\psi _{m}\right) $
is Gaussian. A generalized Gaussian process is uniquely determined by its
mean functional, $\mu _{b}:$ \ $(\mu _{b},\psi )=E(b,\psi ),$ and the
covariance bilinear functional, $B_{b}(\psi _{i},\psi _{j})=E(\left( b,\psi
_{i}\right) \left( b,\psi _{j}\right) ).$

Gelfand, Vilenkin (v.4, p. 260) give the covariance functional of the
generalized derivative, $W^{\prime },$ of the Wiener process as 
\begin{equation*}
B_{W^{\prime }}(\psi _{1},\psi _{2})=\int_{0}^{\infty }\psi _{1}(t)\overline{%
\psi _{2}(t)}dt
\end{equation*}%
where the overbar represents complex conjugation; for real-valued processes
it is not needed. The Fourier transform of a Gaussian random process $b$ is
also a Gaussian random process with covariance functional 
\begin{equation*}
B_{Ft(b)}(\psi _{i},\psi _{j})=E(\left( b,Ft(\psi _{i})\right) \left(
b,Ft(\psi _{j})\right) ).
\end{equation*}%
So for $Ft(W^{\prime })$ the covariance functional is 
\begin{equation*}
B_{Ft(W^{\prime })}(\psi _{1},\psi _{2})=\int_{0}^{\infty }Ft(\psi
_{1})(\zeta )\overline{Ft(\psi _{2})(\zeta )}d\zeta .
\end{equation*}%
The mean functional is zero for $W^{\prime }$ and $Ft(W^{\prime }).$

Gelfand and Vilenkin (1964) provide definitions and results for generalized
random functions in $G^{\ast }=D^{\ast },$ rather than $S^{\ast }.$ One can
similarly define random generalized functions on other spaces of test
functions, not necessarily infinitely differentiable, e.g. on $D_{k}$ of $k$
times continuously differentiable functions with compact support, leading to
space $D_{k}^{\ast }$ of random linear continuous functionals on $D_{k}.$

\subsection{Stochastic convergence of random generalized functions}

A random sequence $b_{n}$ of elements of a space of generalized functions $%
G^{\ast }$ converges to zero in probability: $b_{n}\rightarrow _{p}0$ in $%
G^{\ast },$ or almost surely: $b_{n}\rightarrow _{a.s.}0$ in $G^{\ast }$ if
for any finite set $\psi _{1},...\psi _{v}\in G$ the random vectors (($%
b_{n},\psi _{1}),...(b_{n},\psi _{v}))\rightarrow _{p}(0,...,0)$ or
correspondingly, $\Pr ($(($b_{n},\psi _{1}),...(b_{n},\psi _{v}))\rightarrow
(0,...,0))=1.$

Similarly, convergence in distribution of generalized random processes $%
b_{n}\Rightarrow _{d}b$ is defined by the convergence of all multivariate
distributions for random vectors (($b_{n},\psi _{1}),...(b_{n},\psi
_{v}))\rightarrow _{d}$ $((b,\psi _{1})...(b,\psi _{v}))$ for any set $\psi
_{1},...,\psi _{v}\in G.$

\textbf{Remark 1. }\textit{(a) If }$b_{n}-b\rightarrow _{p}0$\textit{\ in }$%
S^{\ast }$\textit{\ then }$Ft(b_{n})-Ft(b)\rightarrow _{p}0$\textit{\ in }$%
S^{\ast }$\textit{\ and }$Ft^{-1}(b_{n})-Ft^{-1}(b)\rightarrow _{p}0$\textit{%
\ in }$S^{\ast }.$\textit{\ Indeed, for any set }$\psi _{1},...\psi _{v}\in
S $\textit{\ }%
\begin{eqnarray*}
&&((Ft(b_{n})-Ft(b),\psi _{1}),...,(Ft(b_{n})-Ft(b),\psi _{v})) \\
&=&((b_{n}-b,Ft(\psi _{1})),...,(b_{n}-b,Ft(\psi _{v})))
\end{eqnarray*}%
\textit{and since the set }$Ft(\psi _{1}),...,Ft(\psi _{v})\in S$\textit{\
then} $((b_{n}-b,Ft(\psi _{1})),...,(b_{n}-b,Ft(\psi _{v})))\rightarrow
_{p}0.$ \textit{Similarly for }$Ft^{-1}.$

\textit{(b) If }$\mu \in G$\textit{\ and }$b_{n}-b\rightarrow _{p}0$\textit{%
\ in }$G^{\ast },$\textit{\ then }$b_{n}\mu -b\mu \rightarrow _{p}0$\textit{%
\ in }$G^{\ast }.$\textit{\ This follows similarly from the fact that }$\mu
\psi \in G$\textit{\ for }$\psi \in G.$

\textit{(c) Parts (a) and (b) of this Remark also hold with }$\rightarrow
_{a.s.}$\textit{\ replacing }$\rightarrow _{p}$\textit{\ and with
convergence to zero replaced by convergence in distribution to a limit
generalized random process.}

\section{Consistent estimation of solutions to stochastic equations}

Suppose that the known functions, $w$ or $w_{1},w_{2k}$ (equivalently, $%
\varepsilon $ or $\varepsilon _{1},\varepsilon _{2k})$ are consistently
estimated in $S^{\ast }$. In the models discussed here these functions are
density functions or conditional mean functions; typically these are
ordinary functions for which commonly used nonparametric estimators are
shown to be consistent pointwise or in some norm (uniform or $L_{1},$ say)
under some sufficient conditions. If the functions and the estimators can be
represented as elements in $S^{\ast }$ then convergence in many common norms
implies convergence in the weaker topology of $S^{\ast },$ so that such
consistent estimators are consistent in $S^{\ast }$ and the following
discussion of consistency of the solutions to convolution equations
considered here applies. However, consistency in $S^{\ast }$ applies more
widely than the usual consistency results. For example, as shown in
Zinde-Walsh, 2008, kernel estimators of density converge as stochastic
generalized functions even when they diverge pointwise (e.g. at mass points,
or for fractal measures).

For the equations considered here, when identification and well-posedness
holds consistent estimation of the known functions can provide consistent
plug-in estimators for the solutions as long as the estimators all be in the
classes over which well-posedness obtains.

The fact that well-posedness implies consistency of some plug-in estimators
holds generally. Denote (generically) the known functions by $%
w_{1},...w_{M}, $ their consistent in topology of $S^{\ast }$ estimators by $%
\hat{w}_{i};$ denote the solutions by $g_{1},...,g_{K}$ and the plug-in
estimators from solving the equations using $\hat{w}_{i}$ by $\hat{g}_{j}.$
If the estimators, $\hat{w}_{i},$ are in classes where well-posedness holds,
then for any set of $\left( \psi _{1},...\psi _{v}\right) $ from $S$ and any
neighborhood $N_{vK}(0)$ of zero in real or complex $vK-$ dimensional
Euclidean space there is a neighborhood of zero, $\ \tilde{N}_{vM}(0)$ in
the corresponding $vM-$dimensional space such that event $E_{w}=\left\{
(\left( \hat{w}_{1},\psi _{1}\right) ,...,\left( \hat{w}_{1},\psi
_{v}\right) ,\left( \hat{w}_{2},\psi _{1}\right) ,...,\left( \hat{w}%
_{M},\psi _{v}\right) )^{\prime }\in \tilde{N}_{vM}\left( 0\right) \right\} $
by well-posedness implies the event 
\begin{equation*}
E_{g}=\left\{ (\left( \hat{g}_{1},\psi _{1}\right) ,...,\left( \hat{g}%
_{1},\psi _{v}\right) ,\left( \hat{g}_{2},\psi _{1}\right) ,...,\left( \hat{g%
}_{K},\psi _{v}\right) )^{\prime }\in N_{vK}\left( 0\right) \right\} .
\end{equation*}%
Consistency of $\hat{w}_{i}$ as $n\rightarrow \infty $ means that for large
enough $n$ probability of $E_{w}$ can be arbitrarily close to 1, thus
probability of $E_{g}$ is as close or closer to 1. Thus the condition for
consistent plug-in estimation is that the estimators are in the classes of
generalized functions that provide well-posedness of solutions with
probability approaching 1.

\subsection{Consistency of deconvolution: solving $\left( \protect\ref{eq1}%
\right) $}

The well-posedness condition directly applies to the measurement error
deconvolution problem where $\phi $ is known and the conditions apply only
to estimators of observables. The requirement of boundedness of the sequence
of the estimators in the space $S^{\ast }$ would follow if the norms of the
functions were bounded in probability; $L_{p}$ are the spaces typically used
in the literature (e.g. Fan, 1991; Carrasco and Florens, 2010). But here
more generally consistency of plug-in deconvolution holds for any
probability distribution; the only important restriction is on $\phi ^{-1}$
to satisfy $\left( \ref{conds}\right) ;$ that is for the measurement error
not to be supersmooth.

Moreover, the functions in $\left( \ref{conv}\right) $ and the corresponding 
$\left( \ref{eq1}\right) $ need not be generalized densities and
characteristic functions; the conditions for consistent estimation follow
from those for well-posedness and are essentially that the two functions
belong to a convolution pair (equivalently, Fourier transforms to a product
pair) and that the known function be continuous and $\phi ^{-1}$ satisfy $%
\left( \ref{conds}\right) .$

For example, consider in the univariate case a generalized density function, 
$w.$ One can use empirical distribution functions to estimate the
corresponding distribution function, then the estimator of generalized
density by generalized derivatives provides as an estimator the sum of
delta-functions: $\hat{w}_{n}=\frac{1}{n}\tsum\limits_{j=1}^{n}\delta
(x_{j}),$ where $\left( \delta (x_{j}),\psi \right) =\psi (x_{j}).$ Then the
corresponding Fourier transform $\hat{\varepsilon}_{n}=Ft(\hat{w}_{n})$ is
given by a continuous function $\hat{\varepsilon}_{n}(s)=\frac{1}{n}%
\tsum\limits_{j=1}^{n}e^{is^{T}x_{j}}.$ Deconvolution using the
deconvolution kernel (as e.g. in An and Hu, 2012) is applied when the
density of the mismeasured variable is assumed to exist and belong in some $%
L_{p},$ the deconvolution kernel incorporates spectral cut-off by employing
estimators of the form $\tilde{\varepsilon}_{n}=\hat{\varepsilon}%
_{n}(s)I(\left\vert s\right\vert <T_{n}),$ where $T_{n}$ goes to $\infty $
as $n\rightarrow \infty $ at some rate; the indicator function allows for
smoothing in the inverse Fourier transform.

By the results here consistency in $S^{\ast }$ is ensured for both
estimators. Indeed since empirical distribution converges uniformly in
probability to the distribution function, it thus converges in $S^{\ast },$
then the generalized derivatives also converge in $S^{\ast }$ and by
continuity of the Fourier transform estimators $\hat{\varepsilon}_{n}$ and
for $T_{n}\rightarrow \infty $ also $\tilde{\varepsilon}_{n}$ converge to $%
\varepsilon $ in probability. All these functions are continuous and bounded
and only $(\ref{conds})$ for $\phi ^{-1}$ is needed for consistency of the
deconvolution solution in $S^{\ast }$ and the spectral cut-off is not
required. The following remark summarizes this result.

\textbf{Remark 2.} \textit{If the functions in equation }$\left( \ref{conv}%
\right) $\textit{\ are all generalized density functions, the known
characteristic function }$\phi $\textit{\ is such that }$\phi ^{-1}$\textit{%
\ satisfies }$(\ref{conds}),$\textit{\ }$\left\{ z_{i}\right\} _{i=1}^{n}$%
\textit{\ is a random sample from the distribution with the generalized
density }$w,$\textit{\ the deconvolution estimator }%
\begin{equation*}
\hat{g}_{n}=Ft^{-1}\left( \phi ^{-1}\hat{\varepsilon}_{n}\right) \text{
where }\hat{\varepsilon}_{n}(s)=\frac{1}{n}\tsum%
\limits_{j=1}^{n}e^{is^{T}z_{j}}
\end{equation*}%
\textit{is consistent in }$S^{\ast }:$\textit{\ for any finite set }$\psi
_{1},...\psi _{v}\in S$\textit{\ \ the random vector }$((\hat{g}_{n}-g,\psi
_{1}),...,\left( \hat{g}_{n}-g,\psi _{v}\right) )^{T}\rightarrow _{p}0.$

Regularization with suitable spectral cut-off in the case of supersmooth
error distributions typically provides a sequence of estimators that will
converge in norm, albeit slowly for supersmooth distributions (Fan, 1991);
convergence to the limit in the normed space implies convergence in the
topology of $S^{\ast }$.\thinspace\ However, if $\phi $ is not supersmooth
the class of generalized functions where consistent estimation holds does
not include all probability distributions. Indeed spectral cut-off
estimation provides generalized densities that are a limit of inverse
Fourier transforms of truncated generalized functions. Schwartz (1964,
pp.271-273) gives a characterization of any function in $S^{\ast }$ with
Fourier transform that has compact support (in a cube $\left\vert
z_{k}\right\vert <C,k=1,...d$) based on Wiener-Paley theorem. Such a
function is a continuous function $g$ that can be extended to a entire
analytic function $G$ of a complex argument and is of exponential type $\leq
2C$, meaning

\begin{equation*}
\underset{|z|\rightarrow \infty }{lim}\frac{log|G(z)|}{\left\vert
z_{1}\right\vert +...+\left\vert z_{d}\right\vert }\leq 2C.
\end{equation*}%
Thus as long as $g$ is such a function or a limit of such functions it can
be expressed via the regularized solution. As in Schwartz the subspace of
all functions of exponential type (for any finite C) can also be considered.
However, if $g$ does not belong to this subspace regularized solutions may
not converge to it even in $S^{\ast }.$

\subsection{Consistency of estimated solutions to the system of equations $%
\left( \protect\ref{newey2}\right) $ and $\left( \protect\ref{eq2}\right) $}

Establishing consistency for the system of equations is more complicated
since consistency requires not only conditions on estimators of known
functions of the observables, $\varepsilon _{\cdot },$ but also verification
that the resulting estimators of $\phi $ satisfy the conditions of
well-posedness in Theorem 5.

The next Theorem gives conditions for a consistent plug-in nonparametric
estimator for the model $\left( \ref{a}-\ref{c}\right) $ that leads to the
system $\left( \ref{newey2}\right) $ and thus (\ref{eq2}) with continuously
differentiable $\phi .$

Semiparametric generalized method of moments estimation of this model for a
class of regression functions that includes functions in $L_{1}(R^{d})$ was
proposed by Wang and Hsiao (2010); semiparametric estimation was also
discussed for somewhat different classes of univariate parametric regression
functions in Schennach (2007) and Zinde-Walsh (2009). Semiparametric
estimation of polynomial regression is in Hausman, Newey, Ichimura and
Powell (1991).

Start by formally stating the assumptions.

\textbf{Assumption 4 (model)}. \textit{In the model }$\left( \ref{a}-\ref{c}%
\right) $\textit{\ in }$R^{d}$\textit{\ the moments of the errors }$%
E(u_{y}|z,u,u_{x})=0;$\textit{\ }$E(u_{x}|z,v)=0;$\textit{\ }$z$\textit{\ is
independent of }$u;$\textit{\ }$g$\textit{\ \ is a generalized function in }$%
S^{\ast }.$

This assumption implies that if the function $g$ is a regular locally
summable function it cannot grow at infinity at a rate faster than some
polynomial rate. This assumption does not exclude the possibility that $g$
has some singularities, for example, it could be a sum of $\delta -$%
functions or a mix of a regular function and some peaks represented by $%
\delta $-functions.

Denote by $f_{z}$ the density of $z.$ Recall that $f$ denoted the
generalized density of $u$ and $\gamma =Ft\left( g\right) ;\phi =Ft\left(
f\right) ,\varepsilon _{\cdot }=Ft(w_{\cdot }).$

\textbf{Assumption 5 (support).}

\textit{(a) supp}$\left( \gamma \right) $\textit{\ is a connected set that
includes }$0$\textit{\ as an interior point;}

(b) \textit{supp}$\left( \phi \right) =R^{d}.$

\textit{(c) supp}$\left( f_{z}\right) =R^{d}.$

Assumption 5(a) would not be satisfied by a polynomial regression function
when support of $\gamma $ consists of one point 0; semiparametric estimation
for that case was provided in Hausman et al (1991). Support assumptions (b)
and (c) are standard; they could be relaxed.

Recall that here $\phi $ is a characteristic function; $\phi (0)=1.$

\textbf{Assumption 6 (generalized functions)}.

\textit{(a) The function }$f_{z}$\textit{\ is continuous and }$\phi $\textit{%
\ belongs to }$O_{M}.$

\textit{(b) The continuous functions }$w_{\cdot }$\textit{\ belong to a
bounded set in }$S^{\ast },$\textit{\ }$S_{0,m}^{\ast }(V);$\textit{\ also, }%
$\phi _{k}^{\prime }$\textit{\ for }$k=1,...,d,$\textit{\ as well as }$\phi
^{-1}$\textit{\ and }$f_{x}^{-1}$ \textit{all belong to }$S_{0,m}^{\ast
}(V). $

\textit{(c) The regression function }$g$\textit{\ is absolutely integrable.}

\textit{(c') The regression function }$g$ \textit{can be represented as a
sum }$g=g_{L1}+g_{g},$\textit{\ where }$g_{L1\text{ }}$\textit{is absolutely
integrable and }$g_{g}$\textit{\ is such that the Fourier transform of }$%
g_{g}$\textit{\ in }$S^{\ast },$\textit{\ }$\gamma _{g}=Ft(g_{g}),$\textit{\
is singular and thus has support set }$\Lambda _{g}$\textit{\ that is
compact and of zero Lebesgue measure; }$\Lambda _{g}$\textit{\ is a proper
subset of supp}$\left( \gamma \right) $\textit{. The generalized function }$%
\gamma _{g}$\textit{\ is such that there exists a deterministic sequence of
regular functions, }$\left( \gamma _{g}\right) _{n}$\textit{\ that converges
to }$\gamma _{g}$\textit{\ in }$S^{\ast }$\textit{\ and such that support of 
}$\left( \gamma _{g}\right) _{n}$\textit{\ is in a compact set }$\Lambda
_{gn};$\textit{\ there exists }$\zeta _{0}>0$\textit{\ such that }$%
\left\vert \left( \gamma _{g}\right) _{n}\right\vert >2\zeta _{0}$\textit{\
and for }$\gamma _{L1}=Ft(g_{L1})$\textit{\ on }$\Lambda _{g}$ $\left\vert
\gamma _{L1}-\left( \gamma _{g}\right) _{n}\right\vert >\zeta _{0}$\textit{.}

Assumptions 6(a) and 6(b) imply that $w_{\cdot }$ can be divided by $f_{z}$
and any generalized function can be divided by $\phi .$ Requiring that $\phi 
$ be in $O_{M}$ is sufficient to ensure that the model leads to equations $%
\left( \ref{newey2}\right) $ in $S^{\ast }.$ More detailed conditions
similar to those employed in Theorem 2 would allow relaxing the infinite
differentiability assumption 6(a). In particular, if $\gamma $ is a
characteristic function Assumption 6(a) for $\phi $ is not needed.

Continuity of $w_{\cdot }$ in 6(b) would follow by properties of convolution
if either $g$ were continuously differentiable, or $f$ were continuous.

In (b) using the same bound on growth, $\ m,$ and the same $V\,\ $\ for all
the functions simplifies exposition without loss of generality. The bounds
could be liberal but are assumed known in the construction of estimators.
The constraint on the $\phi ^{-1}$ restricts the measurement error from
being supersmooth and the constraint on $f_{z}^{-1}$ does not permit fast
decline to zero at infinity for the density of conditioning $z;\,$\ these
would be automatically satisfied if supports were bounded.

Assumption (c) implies that $\gamma $ and therefore $\varepsilon _{1}$ are
continuous functions. Indeed an integrable function has a continuous Fourier
transform, $\gamma ,$ and $\varepsilon _{1}=\gamma \phi $ is continuous
since $\phi $ is a characteristic function and thus continuous.

Assumption 6(c') holds more generally, e.g. if $g$ is a sum of an integrable
function and a polynomial or a $\sin $ or $\cos $ function; in such cases $%
\gamma $ is a sum of a continuous function (that is separated from zero on
bounded sets within its support) with singular functions such as the $\delta
-$function, shifted $\delta -$function and its derivatives. The support
conditions in the assumption imply that the open support set of the
continuous $\gamma _{L1}$ contains $\Lambda _{g}$. The existence of a
function sequence converging in $S^{\ast }$ to $\gamma _{g}$ with the
support properties stated follows from the general properties of generalized
functions; the only substantive condition there is that the values of the
approximating functions $\left( \gamma _{g}\right) _{n}$ be separated from
zero and from the values of $\gamma _{L1}.$ For example, if $g$ included an
additive constant, then $\gamma _{g}$ is a $\delta -$function, and $\left(
\gamma _{g}\right) _{n}$ could be selected as a sequence of step functions.
These approximating functions need not be specified, only their existence is
required for the proof.

The next assumption is on the stochastic properties of the data generating
process, the sampling and on the kernel and bandwidth.

\textbf{Assumption 7.} \textit{Moments of order }$q>d+1$\textit{\ of }$\frac{%
1}{(1+z^{2})^{m}}y,$\textit{\ }$\frac{1}{(1+z^{2})^{m}}xy$\textit{\
conditional on }$z$\textit{\ are bounded; }$\left\{
x_{i},y_{i},z_{i}\right\} _{i=1}^{n}$ is a random sample from $\left\{
x,y,z\right\} ;$ \textit{the kernel }$K$\textit{\ is the indicator function
of the unit sphere; the bandwidth }$h$\textit{\ is such that }$h\rightarrow
0 $\textit{\ and satisfies }$n^{1-\frac{d}{q-1}}h^{d}\rightarrow \infty .$

Note that Assumption 6(b) implies boundedness of first conditional moments
(functions $w_{\cdot })$ of $\frac{1}{(1+z^{2})^{m}}y,$\textit{\ }$\frac{1}{%
(1+z^{2})^{m}}xy.$

Define the estimators (denoting by $\left( x_{k}\right) _{i}$ the $ith$
observation on the $kth$ component of vector $x)$%
\begin{eqnarray}
\hat{w}_{1}\left( z\right) _{n} &=&\frac{\tsum\limits_{i=1}^{n}y_{i}K\left( 
\frac{z_{i}-z}{h}\right) }{\tsum\limits_{i=1}^{n}K\left( \frac{z_{i}-z}{h}%
\right) };  \label{NW} \\
\hat{w}_{2k}\left( z\right) _{n} &=&\frac{\tsum\limits_{i=1}^{n}y_{i}\left(
x_{k}\right) _{i}K\left( \frac{z_{i}-z}{h}\right) }{\tsum\limits_{i=1}^{n}K%
\left( \frac{z_{i}-z}{h}\right) },k=1,...d.  \notag
\end{eqnarray}

By assumption 6(b) the functions $\frac{1}{(1+z^{2})^{m}}w_{1}(z),$\ $\frac{1%
}{(1+z^{2})^{m}}w_{2k}(z)$\ are bounded in absolute value by $V$; by 6(a)
the density of $z,$\ $f_{z}(z),$\ exists and is continuous, then that for $z$%
\ in any closed sphere $S(x,r)\in $supp($f_{z})$\ the ess$\inf f_{z}(z)>0.$\
Together with Assumption 7 this is sufficient to ensure that estimators $%
\frac{1}{(1+z^{2})^{m}}\hat{w}_{1}(z)_{n},$ $\frac{1}{(1+z^{2})^{m}}\hat{w}%
_{2k}(z)_{n}$, of $\frac{1}{(1+z^{2})^{m}}w_{1}(z),$ $\frac{1}{(1+z^{2})^{m}}%
w_{2k}(z),$ with $\hat{w}_{\cdot }(z)_{n}$ computed as $\left( \ref{NW}%
\right) $, with $K$ and $h$ that satisfy assumption 7 converge in
probability uniformly over any compact set (Devroye, 1978). If $m=0$ the
moment condition is a usual condition made for the Nadaraya-Watson estimator
of $w_{\cdot }$; here essentially just the growth of the conditional moment
functions has to be restricted; this provides estimators that converge in
the topology of $S^{\ast }.$ The bound $V$ is assumed known here; more
restrictive assumptions, including in particular differentiability of $%
w_{\cdot }$ could provide a uniform over compact sets rate of convergence
for e.g. asymptotically optimal estimators (e.g. Stone, 1982); then $V$ that
defines $\tilde{w}_{\cdot n}$ could grow with sample size.

Uniform convergence in probability implies that the estimators converge in
probability in the topology of $D^{\ast }.$ If the functions have compact
support this implies convergence in $S^{\ast }$ since on compact support it
coincides with convergence in $D^{\ast },$ but on unbounded support growth
at infinity needs to be controlled for convergence in $S^{\ast }.$ Thus, for
any generic estimator $\hat{w}_{n}\,$ represented by a regular function, $%
\hat{w}(z)_{n},$ of a regular generalized function $w=w\left( z\right) $ in $%
S_{1,m}^{\ast }\left( V\right) $ define a corresponding estimator $\tilde{w}%
_{n}=\tilde{w}\left( z\right) _{n}$ by setting it to $\hat{w}(z)_{n}$ if $|%
\hat{w}(z)_{n}|<V(1+z^{2})^{m},$ and $V(1+z^{2})^{m}$ otherwise.

\textbf{Lemma 3.} \textit{Suppose that the estimator }$\hat{w}_{n}$\textit{\
converges to }$w$\textit{\ in probability uniformly on bounded sets; then }$%
\tilde{w}_{n}$\textit{\ converges to }$w$\textit{\ in probability in }$%
S^{\ast }$\textit{\ and the corresponding Fourier transform }$\hat{%
\varepsilon}_{n}=Ft(\tilde{w}_{n})$\textit{\ converges in probability in }$%
S^{\ast }$\textit{\ to }$\varepsilon =Ft(w).$

\bigskip \textbf{Proof.} Consider a set $\psi _{1},...,\psi _{v}\in S.$ For
any $\zeta >0$ find a compact set $\Lambda $ such that for $\tilde{\psi}$
representing any of $\psi _{1},...,\psi _{v},Ft(\psi _{1}),...,Ft(\psi _{v})$%
\begin{equation*}
\int_{R^{d}\smallsetminus \Lambda }V(1+t^{2})^{m}|\tilde{\psi}(t)|dt<\zeta .
\end{equation*}

Consider $\left\vert \tilde{w}_{n}-w\right\vert \leq \left\vert \hat{w}%
_{n}-w\right\vert I\left( z\in \Lambda \right) +\left\vert \tilde{w}_{n}-%
\hat{w}_{n}\right\vert I\left( z\in \Lambda \right) +\left\vert \tilde{w}%
_{n}-w\right\vert I\left( z\in R^{d}\diagdown \Lambda \right) .$

Then for any $\zeta _{1}$ there exists $N$ such that $\Pr \left( \left\vert
\left( \tilde{w}_{n}-w,\tilde{\psi}\right) \right\vert >\zeta _{1}\right)
\leq $ 
\begin{equation*}
\Pr \left( \underset{\Lambda }{\sup }\left\vert \hat{w}_{n}-w\right\vert
|\int \tilde{\psi}(t)dt|>\zeta _{1}\right) +\Pr \left( z\in \Lambda
:\left\vert \tilde{w}_{n}-\hat{w}_{n}\right\vert |\int \tilde{\psi}%
(t)dt|>\zeta _{1}-2\zeta \right) \leq \zeta _{1},
\end{equation*}%
since the value $\left\vert \left( \tilde{w}_{n}-w,\tilde{\psi}\right)
\right\vert $ outside $\Lambda $ is bonded with certainty by 2$\zeta $ and
the two probabilities on bounded $\Lambda $ can be bounded because of
uniform convergence in probability of the estimators. Recall that for
Fourier transforms $\left( Ft(\tilde{w}_{n}-w),\psi \right) =\left( \hat{w}%
_{n}-\hat{w},Ft(\psi ))\right) .$ Thus convergence in probability in $%
S^{\ast }$ of the estimators of the functions $\hat{w},$ and also of the
Fourier transforms $\varepsilon $ is established.$\blacksquare $

Denote $Ft(\tilde{w}_{\cdot n})$ by $\hat{\varepsilon}_{\cdot n},$ note that
by construction these random generalized functions are infinitely
differentiable since by the assumption on the kernel $K$ the support of $%
\hat{w}_{\cdot n}$ is bounded and thus the Fourier transform is
differentiable. The Theorem below establishes consistency of plug-in
estimators.

\begin{theorem}
\textit{If Assumptions 4-7 are satisfied, then }

\textit{(i) if supp(}$\gamma )$ \textit{is bounded\ the plug-in estimator }%
\begin{equation*}
\hat{\gamma}_{n}=\left[ \exp (-\int_{0}^{s}\tsum_{k=1}^{d}\hat{\varepsilon}%
_{1n}^{-1}(t)(\left( \hat{\varepsilon}_{1n}(t)\right) _{k}^{\prime }-i\hat{%
\varepsilon}_{2kn}(t))dt_{k}\right] \hat{\varepsilon}_{1n}
\end{equation*}%
\textit{is such that it exists with probability going to 1 and }$Ft^{-1}(%
\hat{\gamma}_{n})-g\rightarrow _{p}0$\textit{\ in }$S^{\ast };$

\textit{(ii) generally the estimator }$\hat{\gamma}_{n}=\tilde{\phi}^{-1}%
\hat{\varepsilon}_{1n}$ \textit{with }%
\begin{equation*}
\tilde{\phi}^{-1}=\tilde{\phi}_{n}^{-1}\left( s\right) =\left[ \exp
(-\int_{0}^{s}\tsum_{k=1}^{d}\hat{\varepsilon}_{1n}^{-1}(t)(\left( \hat{%
\varepsilon}_{1n}(t)\right) _{k}^{\prime }-i\hat{\varepsilon}_{2kn}(t))dt_{k}%
\right] ,
\end{equation*}%
\textit{if }%
\begin{equation*}
\left\vert \left[ \exp (-\int_{0}^{s}\tsum_{k=1}^{d}\hat{\varepsilon}%
_{1n}^{-1}(t)(\left( \hat{\varepsilon}_{1n}(t)\right) _{k}^{\prime }-i\hat{%
\varepsilon}_{2kn}(t))dt_{k}\right] \right\vert <V\left( 1+s^{2}\right) ^{m}
\end{equation*}%
\textit{and}%
\begin{equation*}
V\left( 1+s^{2}\right) ^{m}
\end{equation*}%
\textit{otherwise,} \textit{is such that it exists with probability
approaching 1 and }$Ft^{-1}(\hat{\gamma}_{n})-g\rightarrow _{p}0$\textit{\
in }$S^{\ast }.$
\end{theorem}

\begin{proof}
The proof consists of the following steps.

Step 1 considers a compact set inside the support of $\gamma $ and shows
that on such a set consistency in $S^{\ast }$ follows.

Step 2 examines the continuous part of $\gamma ,$ $\gamma _{L1},$ and
convergence is probability of estimators defined on compact sets that
exclude a set containing the support of the singular part.

Combining the results of Step 1 and 2 concludes the proof of (i).

Step 3 considers the case of unbounded support with $\tilde{\phi}_{n}^{-1}$
defined in (ii). On any compact set with this estimator replacing $\hat{\phi}%
_{n}^{-1}$ the results in Steps 1 and 2 hold so consistency on a bounded set
obtains. Consistency on the unbounded support is shown in topology of $%
S^{\ast }$ by selecting for any set of $\psi _{1},...\psi _{v}$ from $S$ the
corresponding compact set defined in Lemma 3 to bound all the functionals in
probability outside of the compact set to prove (ii).

Next, the details are provided.

Step 1. First, $\hat{\varepsilon}_{1n};$ this is a sequence of continuous
functions that converge in probability in $S^{\ast }$ by Lemma 3. Fourier
transforms of functions in $L_{1}$ are uniformly continuous on compact sets.
If Assumption 6(c) holds convergence is to the continuous function $%
\varepsilon _{1}$. Then $\hat{\varepsilon}_{1n}$ converge in probability
pointwise and uniformly on the compact set $\Lambda .$ Then for any $0<\zeta
_{1},\zeta _{2}<1$ we can find $N_{1}\equiv N(\Lambda ,\zeta _{1},\zeta
_{2}) $ such that for $n>N_{1}$%
\begin{equation*}
\Pr (\underset{\Lambda }{\sup }\left\vert \hat{\varepsilon}_{1n}-\varepsilon
_{1}\right\vert >\zeta _{1}^{3})<\zeta _{2}.
\end{equation*}%
Consider $\Lambda $ that satisfies $\underset{\Lambda }{\inf }\left\vert
\varepsilon _{1}\right\vert >0.$ Set $\zeta _{1}\leq \underset{\Lambda }{%
\inf }\left\vert \varepsilon _{1}\right\vert $, then 
\begin{eqnarray}
\Pr (\underset{\Lambda }{\inf }\left\vert \hat{\varepsilon}_{1n}\right\vert
&<&\zeta _{1})<\zeta _{2};\text{ and}  \label{e1inverse} \\
\Pr (\underset{\Lambda }{\sup }\left\vert \hat{\varepsilon}%
_{1n}^{-1}\right\vert &>&\zeta _{1}^{-1})<\zeta _{2}.  \notag
\end{eqnarray}

Under 6(c') for a singular $\gamma _{g}$ there exists a deterministic
sequence of regular functions, $\left( \gamma _{g}\right) _{n}$ that
converges to $\gamma _{g}$\ in $S^{\ast }$\ and such that support of $\left(
\gamma _{g}\right) _{n}$ is in a compact set $\Lambda _{gn}$\ that is a
proper subset of support of $\gamma ;$ it can be selected such that for some
sequence $\zeta _{gn}\rightarrow 0$ the Lebesgue measure of $\Lambda _{gn},$ 
$\lambda \left( \Lambda _{gn}\right) <\zeta _{gn};$ $\Lambda
_{gn_{1}}\subset \Lambda _{gn_{2}}$ for $n_{1}>n_{2}$ and on any $\Lambda
_{gn}$ for some fixed $\zeta _{0}$ we get $\underset{\Lambda _{gn}}{\inf }%
\left\vert \left( \gamma _{g}\right) _{n}\right\vert >\zeta _{0}.$ For the
corresponding sequence $\left( \varepsilon _{g}\right) _{n}=\left( \gamma
_{g}\right) _{n}\phi $ on $\Lambda _{gn}$ positive lower bounds on modulus
exist and are no less than $\zeta _{0}\underset{\Lambda _{gn}}{\inf }%
\left\vert \phi \right\vert $.

Under 6(c') consider the sequence of deterministic functions $\left(
\varepsilon _{g}\right) _{n}=\left( \gamma _{g}\right) _{n}\phi $ and the
difference $\left( \varepsilon _{L1}\right) _{n}=\varepsilon _{1}-\left(
\varepsilon _{g}\right) _{n}.$ This is a sequence of piece-wise
deterministic continuous functions that converges in $S^{\ast }$ to the
continuous function $\varepsilon _{1L1}=\gamma _{L1}\phi .$ Then on a
compact set $\Lambda $ for $\zeta _{1}\leq \min \left( \underset{\Lambda }{%
\inf }\left\vert \varepsilon _{1L1}\right\vert ,\text{ }\zeta _{0}\underset{%
\Lambda _{gn}}{\inf }\left\vert \phi \right\vert \right) $ we can find the
corresponding $N$ so that for $\hat{\varepsilon}_{1n}=\hat{\varepsilon}%
_{1n}-\left( \varepsilon _{g}\right) _{n}+\left( \varepsilon _{g}\right)
_{n} $ 
\begin{equation*}
\underset{\Lambda }{\Pr (\sup }\left\vert \hat{\varepsilon}_{1n}-\left(
\varepsilon _{g}\right) _{n}-\varepsilon _{1L1}\right\vert <\zeta
_{1})<\zeta _{2}.
\end{equation*}

Then 
\begin{equation*}
\Pr (\underset{A}{\inf }\left\vert \hat{\varepsilon}_{1n}\right\vert <\zeta
_{1})<\Pr (\underset{A}{\inf }\left\vert \hat{\varepsilon}_{1n}-\left(
\varepsilon _{g}\right) _{n}\right\vert <\zeta _{1})<\zeta _{2};
\end{equation*}%
and $\left( \ref{e1inverse}\right) $ also holds.

Then with probability approaching 1 on compact $\Lambda $ the continuous
random functions $\hat{\varkappa}_{kn}=\hat{\varepsilon}_{1n}^{-1}(\left( 
\hat{\varepsilon}_{1n}\right) _{k}^{\prime }-i\hat{\varepsilon}_{2kn}),$ $%
k=1,...,d$ are in a bounded set in $S^{\ast },$ and thus $\hat{\phi}%
_{n}^{-1}(s)=\exp (-\int_{0}^{s}\tsum {}_{k=1}^{d}\hat{\varkappa}%
_{k}(t)dt_{k})$ is also in a bounded set, and thus with probability
approaching 1 satisfies the condition for well-posedness of Theorem 5. Then
the estimators $\hat{\gamma}_{n}=\hat{\phi}_{n}^{-1}\varepsilon _{1}$ are
consistent for $\gamma $ in $S^{\ast }$ on the compact set $\Lambda $ where $%
\underset{\Lambda }{\inf }\left\vert \varepsilon _{1L1}\right\vert >0.$

Since $\gamma _{L1}$ is a continuous function its support is an open set. By
Assumption 6(c') the compact support of $\gamma _{g}$ is contained inside
the open support of $\gamma _{L1}.$ If $\underset{\text{supp}\left( \gamma
_{l1}\right) }{\inf }\left\vert \varepsilon _{1L1}\right\vert >\zeta _{1}>0$
this concludes the proof of (i).

Step 2. Otherwise consider the set $\bar{\Lambda}=\cup \Lambda _{gn};$ and
the open set $\Omega =$supp($\gamma )\diagdown \bar{\Lambda};\,$\ then $\hat{%
\varepsilon}_{1n}$ converge to $\varepsilon _{1}$ that is continuous on that
set. Of course under Assumption 6(c) $\Omega =supp(\gamma ).$ Generally $%
\Omega $ is a union of open connected sets and we can proceed by considering
each component in $\Omega .$ Without loss of generality such a component can
be assumed to be a connected open set containing zero as an interior point,
since if it does not contain zero by a shift (which is a continuous
operation in $S^{\ast })$ of an arbitrary interior point into zero this can
be attained.

All the proofs that follow apply to an open set that will be denoted $\Omega 
$ that is connected and contains zero as an interior point.

Consider a compact set $\Lambda \subset \Omega $ where $\underset{\Lambda }{%
\inf }\left\vert \varepsilon _{1}\right\vert >0.$ It follows from the proof
in step 1 under continuity of $\varepsilon _{1}$ that $\Pr \left( \underset{%
\Lambda }{\sup }\left\vert \hat{\varepsilon}_{1n}^{-1}-\varepsilon
_{1}^{-1}\right\vert >\zeta _{1}\right) =$ 
\begin{eqnarray*}
&&\Pr \left( \underset{\Lambda }{\sup }\left\vert \left( \hat{\varepsilon}%
_{1n}^{-1}(\varepsilon _{1}-\hat{\varepsilon}_{1n}\right) \varepsilon
_{1}^{-1}\right\vert >\zeta _{1}\right) \\
&\leq &\Pr \left( \underset{\Lambda }{\sup }\left\vert \left( \hat{%
\varepsilon}_{1n}^{-1}(\varepsilon _{1}-\hat{\varepsilon}_{1n}\right)
\right\vert >\zeta _{1}^{2}\right) \\
&\leq &\Pr \left( \underset{\Lambda }{\sup }\left\vert \hat{\varepsilon}%
_{1n}^{-1}\right\vert >\zeta _{1}^{-1},\underset{\Lambda }{\sup }\left\vert 
\hat{\varepsilon}_{1n}^{-1}(\varepsilon _{1}-\hat{\varepsilon}%
_{1n})\right\vert >\zeta _{1}^{2}\right) \\
&&+\Pr \left( \underset{\Lambda }{\sup }\left\vert \hat{\varepsilon}%
_{1n}^{-1}\right\vert \leq \zeta _{1}^{-1},\underset{\Lambda }{\sup }%
\left\vert \hat{\varepsilon}_{1n}^{-1}(\varepsilon _{1}-\hat{\varepsilon}%
_{1n})\right\vert >\zeta _{1}^{2}\right) \\
&\leq &\Pr \left( \underset{\Lambda }{\sup }\left\vert \hat{\varepsilon}%
_{1n}^{-1}\right\vert >\zeta _{1}^{-1}\right) +\Pr \left( \underset{\Lambda }%
{\sup }\left\vert \varepsilon _{1}-\hat{\varepsilon}_{1n}\right\vert >\zeta
_{1}^{3}\right) \\
&\leq &2\zeta _{2}.
\end{eqnarray*}

Theorem 3(b) implies that $\phi ^{-1}(s)=\exp (-\int_{0}^{s}\tsum
{}_{k=1}^{d}\varkappa _{k}(t)dt_{k}),$ where $\varkappa _{k}$ is the unique
continuous function that solves $\varepsilon _{1}\varkappa -(\left(
\varepsilon _{1}\right) _{k}^{\prime }-i\varepsilon _{2k})=0$ in $S^{\ast }$
on $\Omega .$ By Assumptions 6(a-b) for $\phi $, $\varepsilon _{1}\phi ^{-1}$
exists in $S^{\ast }$and $g=Ft^{-1}(\varepsilon _{1}\phi ^{-1})$ in $S^{\ast
}.$ By Assumption 5a $\varepsilon _{1}=\gamma \phi $ is non-zero on supp($%
\gamma $).

Consider the estimator function $\hat{\varkappa}_{kn}=\hat{\varepsilon}%
_{1n}^{-1}(\left( \hat{\varepsilon}_{1n}\right) _{k}^{\prime }-i\hat{%
\varepsilon}_{2kn})$ on compact $\Lambda \subset \Omega ;$ the function $%
\varepsilon _{1}$ is continuous there and it follows that $\left(
\varepsilon _{1}\right) _{k}^{\prime }-i\varepsilon _{2k}=\varepsilon
_{1}\varkappa _{k}$ is continuous.

The sequences of random functions $(\left( \hat{\varepsilon}_{1n}\right)
_{k}^{\prime }-i\hat{\varepsilon}_{2kn})$ converge in probability uniformly
on $\Lambda $ in $S^{\ast }$ to the continuous function $(\left( \varepsilon
_{1}\right) _{k}^{\prime }-i\varepsilon _{k2}).$ Define $\bar{B}=\underset{%
\Lambda ,k}{\sup }\left\vert \left( \varepsilon _{1}\right) _{k}^{\prime
}-i\varepsilon _{k2}\right\vert .$ For $0<\zeta _{4}$ find $N_{2}$ such that 
\begin{equation*}
\Pr (\underset{\Lambda }{\sup }\left\vert (\left( \hat{\varepsilon}%
_{1n}\right) _{k}^{\prime }-i\hat{\varepsilon}_{2kn})-(\left( \varepsilon
_{1}\right) _{k}^{\prime }-i\varepsilon _{k2})\right\vert >\zeta _{4})<\zeta
_{2}\text{ for }n>N_{2}.
\end{equation*}
Bound $\Pr (\underset{\Lambda }{\sup }\left\vert \hat{\varkappa}%
_{kn}-\varkappa _{k}\right\vert >\zeta _{5})\leq $ 
\begin{equation*}
\Pr (\underset{\Lambda }{\sup }\left\vert \hat{\varepsilon}%
_{1n}^{-1}\right\vert \left\vert (\left( \hat{\varepsilon}_{1n}\right)
_{k}^{\prime }-i\hat{\varepsilon}_{2n}-\left( \varepsilon _{1}\right)
_{k}^{\prime }+i\varepsilon _{2}\right\vert +\underset{\Lambda }{\sup }%
\left\vert \hat{\varepsilon}_{1n}^{-1}-\varepsilon _{1}^{-1}\right\vert
\left\vert \left( \varepsilon _{1}\right) _{k}^{\prime }-i\varepsilon
_{k2}\right\vert >\zeta _{5})\leq 
\end{equation*}%
\begin{eqnarray*}
\Pr (\underset{\Lambda }{\sup }\left\vert \hat{\varepsilon}%
_{1n}^{-1}\right\vert  &>&\zeta _{1}^{-1})+\Pr \left( \underset{\Lambda }{%
\sup }\left\vert (\left( \hat{\varepsilon}_{1n}\right) _{k}^{\prime }-i\hat{%
\varepsilon}_{2kn}-\left( \varepsilon _{1}\right) _{k}^{\prime
}+i\varepsilon _{k2}\right\vert >\zeta _{5}\zeta _{1}\right)  \\
+\Pr (\underset{\Lambda }{\sup }\left\vert \hat{\varepsilon}%
_{1n}^{-1}-\varepsilon _{1}^{-1}\right\vert  &>&\zeta _{5}/\bar{B}).
\end{eqnarray*}%
If $\zeta _{5}=\min \left\{ \zeta _{1}\bar{B},\zeta _{4}/\zeta _{1}\right\} $
the probability as $n>\max \left\{ N_{1},N_{2}\right\} $ is less than $%
4\zeta _{2}.$

Then $\Pr (\underset{\Lambda }{\sup }\left\vert \int_{0}^{s}\tsum
{}_{k=1}^{d}\hat{\varkappa}_{kn}(t)dt_{k}-\int_{0}^{s}\tsum
{}_{k=1}^{d}\varkappa _{k}(t)dt_{k}\right\vert >\zeta _{6})\leq $ 
\begin{equation*}
\Pr (\underset{\Lambda }{\sup }\int_{0}^{s}\tsum {}_{k=1}^{d}|\hat{\varkappa}%
_{kn}(t)-\varkappa _{k}(t)|dt_{k}>\zeta _{6})\leq \Pr \left( \underset{%
\Lambda }{\sup }\left\vert \hat{\varkappa}_{kn}(t)-\varkappa
_{k}(t)\right\vert >\zeta _{6}/\mu (\Lambda )\right) ,
\end{equation*}%
where $\mu (\Lambda )$ is the measure of the compact set $\Lambda .$ For $%
\zeta _{6}=\mu (\Lambda )\zeta _{5}$ then the probability is less than $%
4\zeta _{2}.$

Consider now on $\Lambda $ the function $\hat{\phi}_{n}^{-1}(s)=\exp
(-\int_{0}^{s}\tsum {}_{k=1}^{d}\hat{\varkappa}_{kn}(t)dt_{k}).$ Define $%
\tilde{B}=\underset{\Lambda }{\sup }V\left( 1+s^{2}\right) ^{m};$ then $%
\underset{\Lambda }{\sup }\left\vert \phi ^{-1}(s)\right\vert <\tilde{B}.$

Then $\Pr (\underset{\Lambda }{\sup }\left\vert \hat{\phi}_{n}^{-1}-\phi
^{-1}\right\vert >\zeta _{7})\leq $ 
\begin{equation*}
\Pr (\underset{\Lambda }{\sup }\left\vert \int_{0}^{s}\tsum {}_{k=1}^{d}\hat{%
\varkappa}_{kn}(t)dt_{k}-\int_{0}^{s}\tsum {}_{k=1}^{d}\varkappa
_{k}(t)dt_{k}\right\vert >\ln (1+\tilde{B}^{-1}\zeta _{7})),
\end{equation*}%
and is smaller than $4\zeta _{2}$ for $\zeta _{7}=\ln (1+\tilde{B}^{-1}\zeta
_{6}).$

Since for the continuous functions $\underset{\Lambda }{\sup }\left\vert
\phi _{n}^{-1}\hat{\varepsilon}_{1n}-\phi ^{-1}\varepsilon _{1}\right\vert <$
\begin{equation*}
\tilde{B}\underset{\Lambda }{\sup }\left\vert \varepsilon _{1}-\hat{%
\varepsilon}_{1n}\right\vert +\underset{\Lambda }{\sup }\left\vert
\varepsilon _{1}\right\vert \underset{\Lambda }{\sup }\left\vert \phi
_{n}^{-1}-\phi ^{-1}\right\vert 
\end{equation*}%
by similar derivations 
\begin{equation*}
\Pr \left( \underset{\Lambda }{\sup }\left\vert \hat{\phi}_{n}^{-1}\hat{%
\varepsilon}_{1n}-\phi ^{-1}\varepsilon _{1}\right\vert >\zeta _{8}\right)
<5\zeta _{2}
\end{equation*}
if $\zeta _{8}<\min \{\tilde{B}^{-1}\zeta _{1},\zeta _{7}(\underset{\Lambda }%
{\sup }\left\vert \varepsilon _{1}\right\vert )^{-1}.$

If $\Omega $ is bounded then by Assumption 5 $\underset{\Omega }{\sup }%
\left\vert \phi ^{-1}(s)\right\vert $ is uniformly bounded and since $\hat{%
\phi}_{n}^{-1}$ converges to $\phi ^{-1}$ in probability uniformly on any
compact set inside $\Omega $ then also $\hat{\phi}_{n}^{-1}$ is bounded away
from zero and then $\left\vert \hat{\phi}_{n}^{-1}\hat{\varepsilon}%
_{1n}-\phi ^{-1}\varepsilon _{1}\right\vert $ converges in probability to
zero on $\Omega .$ This concludes the proof of (i)$.$

Step 3. If supp($\gamma )$ is unbounded consider $\tilde{\phi}_{n}^{-1}$
defined in (ii). From the proof in step 1 it follows that for large enough $%
N $ the estimator $\tilde{\phi}_{n}^{-1}=\hat{\phi}_{n}^{-1}$ on any compact 
$\Lambda $ with arbitrarily high probability and then $\hat{\varepsilon}_{1n}%
\tilde{\phi}_{n}^{-1}$ converges to $\gamma $ on $\Lambda $ in probability
in $S^{\ast }$.

Consider an arbitrary set $\psi _{1},...,\psi _{v}\in S$ and the
corresponding compact set $\Lambda $ defined by Lemma 3 and show that $\Pr
(\left\vert (\tilde{\gamma}_{n}-\gamma ,\tilde{\psi})\right\vert >\zeta )$
goes to zero.

Since $\hat{\varepsilon}_{1n}-\varepsilon _{1}$ converges to zero in
probability in $S^{\ast }$ by Lemma 3 and since on $\Omega $ this difference
is a continuous function, then also $|\hat{\varepsilon}_{1n}(t)-\varepsilon
_{1}(t)|$ converges to zero in probability in $S^{\ast }$ on $\Omega $. Thus 
\begin{equation*}
\int_{\Omega \diagdown \Lambda }V\left( 1+t^{2}\right) ^{m}|\hat{\varepsilon}%
_{1n}(t)-\varepsilon _{1}(t)|\left\vert \tilde{\psi}(t)\right\vert dt
\end{equation*}%
converges in probability to zero. Then since

\begin{equation*}
\left\vert (\tilde{\gamma}_{n}-\gamma ,\tilde{\psi})\right\vert \leq 
\underset{\Lambda }{\sup }\left\vert \tilde{\phi}_{n}^{-1}\hat{\varepsilon}%
_{1n}-\phi ^{-1}\varepsilon _{1}\right\vert \int_{\Lambda }\left\vert \tilde{%
\psi}(t)\right\vert dt
\end{equation*}%
\begin{eqnarray*}
&&+\int_{\Omega \backslash \Lambda }V\left( 1+t^{2}\right) ^{m}|\varepsilon
_{1}(t)|\left\vert \tilde{\psi}(t)\right\vert dt \\
&&+\int_{\Omega \backslash \Lambda }V\left( 1+t^{2}\right) ^{m}|\hat{%
\varepsilon}_{1n}(t)-\varepsilon _{1}(t)|\left\vert \tilde{\psi}%
(t)\right\vert dt,
\end{eqnarray*}%
it follows that $\Pr (\left\vert (\tilde{\gamma}_{n}-\gamma ,\tilde{\psi}%
)\right\vert >\zeta )\leq $ 
\begin{eqnarray*}
\Pr (\underset{\Lambda }{\sup }\left\vert \hat{\phi}_{n}^{-1}\hat{\varepsilon%
}_{1n}-\phi ^{-1}\varepsilon _{1}\right\vert &>&\left( \int_{\Lambda
}\left\vert \tilde{\psi}(t)\right\vert dt\right) ^{-1}\zeta ) \\
+\Pr \left( \int_{\Omega \backslash \Lambda }V\left( 1+t^{2}\right)
^{m}|\varepsilon _{1}(t)|\left\vert \tilde{\psi}(t)\right\vert dt\right)
&>&\zeta ) \\
+\Pr \left( \int_{\Omega \backslash \Lambda }V\left( 1+t^{2}\right) ^{m}|%
\hat{\varepsilon}_{1n}(t)-\varepsilon _{1}(t)|\left\vert \tilde{\psi}%
(t)\right\vert dt\right) &>&\zeta )
\end{eqnarray*}

Here as shown in the Step 2 the first probability converges to zero, the
second converges to zero by assumption 6 \ on $\varepsilon _{1},$ definition
of the set $\Lambda $ and Lemma 3, and the third by convergence of $\hat{%
\varepsilon}_{1n}.$ Then$\ \tilde{\gamma}_{n}$ converges in probability to $%
\gamma =\phi ^{-1}\varepsilon _{1}$ in $S^{\ast }.$ Taking inverse Fourier
transforms in $S^{\ast }$ concludes the proof.
\end{proof}

The Theorem provides consistency of plug-in estimators for solutions to the
system of equations $\left( \ref{eq2}\right) $ and consequently $\left( \ref%
{newey2}\right) $ in a fairly general set-up. Nevertheless some assumptions
can be further relaxed. Of course, establishing results for a compact
support of the Fourier transforms is much easier and thus using spectral
cut-off can be advantageous especially when high frequency components of the
regression function may be commesurate with the magnitude of the error
components.

Computation of the estimators requires applying Fourier transforms and
inverse Fourier transforms. This can be accomplished with numerical
algorithms. However, it is possible to simplify the estimated $\hat{%
\varepsilon}$. Consider instead of the estimator in $\left( \ref{NW}\right) ,
$ $\hat{w}_{1}\left( z\right) _{n},$ an estimator computed as $%
\tsum\limits_{i=1}^{n}\alpha _{i}^{-1}y_{i}K\frac{z_{i}-z}{h})$ with the
weight $\alpha _{i}=\tsum\limits_{j\neq i}^{n}K(\frac{z_{j}-z_{i}}{h})$
replacing the $z-$dependent weight in $\hat{w}_{1}\left( z\right) _{n}$ and
similar estimators for $\hat{w}_{2k}\left( z\right) _{n}.$ Then the
corresponding Fourier transform $\hat{\varepsilon}_{1n}(s)$ can be expressed
as $\tsum\limits_{i=1}^{n}\alpha _{i}^{-1}y_{i}e^{is^{T}z_{j}}$\textit{sinc}$%
\left( \frac{s^{T}z_{j}}{\pi }\right) ,$ where by definition \textit{sinc}$%
\left( x\right) =\frac{\sin \pi x}{\pi x},$ and similar expressions for $%
\hat{\varepsilon}_{2kn}.$ Further computation for the estimators would have
to be done numerically.

\section{Conclusion}

This paper was devoted to treating a single convolution equation and a
specific system of convolution equations; many statistical models with
various independence conditions give rise to such equations; measurement
error is emphasized here, but equations of this type are also applicable in
other models, such as factor models and panel data models; many examples are
presented in Zinde-Walsh (2012). The results of this paper indicate
conditions for identification and well-posedness when casting these
equations in terms of generalized functions; the generalized functions
approach enlarges the area of applicability of the models.

\end{document}